\documentclass[11pt,a4paper,leqno]{amsart}
\usepackage{float}
\usepackage{amsfonts,amsmath,graphics,epsfig,latexsym,times,amssymb,epic,eepic,color}
\usepackage[all]{xy}
\newif\ifpix \pixtrue
\usepackage[english]{babel}

\usepackage{graphicx,pstricks}

\numberwithin{equation}{section}

\calclayout

\numberwithin{equation}{section}

\def\R{\mathbb{R}}

\def\C{\mathbb{C}}
\def\H{\mathbb{H}}
\def\Z{\mathbb{Z}}
\def\Q{\mathbb{Q}}
\def\T{\mathbb{T}}
\def\P{\mathbb{P}}
\def\D{\mathbb{D}}

\def\k{\mathfrak{k}}
\def\h{\mathfrak{h}}
\def\t{\mathfrak{t}}

\def\om{\omega}
\def\Om{\Omega}

\def\ep{\varepsilon}

\def\deg{\mathrm{deg}}

\def\del{\partial}
\def\delb{\overline{\partial}}

\newtheorem{lemma}[subsubsection]{Lemma}
\newtheorem{prop}[subsubsection]{Proposition}
\newtheorem{cor}[subsubsection]{Corollary}
\newtheorem{theo}[subsection]{Theorem}
\newtheorem{theointro}{Theorem}

\theoremstyle{definition}
\newtheorem{dfn}[subsubsection]{Definition}
\newtheorem{rmk}[subsubsection]{Remark}

\newtheorem*{rmk*}{Remark}
\newtheorem*{rmks*}{Remarks}
\newtheorem{rmks}[subsection]{Remarks}
\renewenvironment{proof}{\paragraph{\emph{Proof}}} {~Q.E.D.\medskip}

\newenvironment{remark*}{\begin{rmk*} --- \normalfont} { \end{rmk*} }
\newenvironment{remarks*}{\begin{rmks*} \begin{enumerate} \normalfont}
{\end{enumerate} \end{rmks*} } 
 
\title[Extremal K\"ahler metrics on blow-ups of parabolic ruled surfaces]{Extremal K\"ahler metrics on blow-ups of parabolic ruled surfaces}
\date{June 2011}
\author{Carl Tipler}
\address{Carl Tipler, Laboratoire Jean Leray LMJL, Nantes France}
\email{carl.tipler@univ-nantes.fr}


\begin{document}

{\Huge \sc \bf\maketitle}
\begin{abstract}
\sloppypar{New examples of extremal K\"ahler metrics are given on blow-ups of parabolic ruled surfaces. 
The method used is based on the gluing construction of Arezzo, Pacard and Singer \cite{aps}. 
This enables to endow ruled surfaces of the form $\P(\mathcal{O}\oplus L)$ with special parabolic structures such that 
the associated iterated blow-up admits an extremal
metric of non-constant scalar curvature.}
\end{abstract}

\section{ Introduction}
\label{secintro}
In this paper is adressed the problem of existence of extremal K\"ahler metrics on ruled surfaces. 
An extremal K\"ahler metric
on a compact K\"ahler manifold $M$ is a metric that minimizes the Calabi functional in a given K\"ahler class $\Om$:

$$
\begin{array}{ccc}
 \lbrace \om\in\Om^{1,1}(M,\R),d\om=0,\; \om >0 \; / [\om]=\Om \rbrace & \rightarrow & \R \\
 \om & \mapsto & \int_M s(\om)^2 \om^n.
\end{array}
$$

Here, $s(\om)$ stands for the scalar curvature of $\om$ and $n$ is the complex dimension of $M$.
Constant scalar curvature metrics are examples of extremal metrics.
If the manifold is polarized by an ample line bundle $L$ 
the existence of such a metric in the class $c_1(L)$
is related to a notion of stability of the pair $(M,L)$. More precisely, the works of Yau \cite{y}, Tian \cite{t},
Donaldson \cite{d} and lastly Sz\'ekelyhidi \cite{s1}, led to the conjecture that a polarized manifold $(M,L)$
admits an extremal K\"ahler metric in the K\"ahler class $c_1(M)$ if and only if it is
relatively K-polystable. So far it has been proved that the existence of a constant scalar curvature K\"ahler metric implies K-stability \cite{mab} and the existence of an extremal metric implies relative K-polystability \cite{ss}. 

We will focus on the special case of complex ruled surfaces.
First consider a geometrically ruled surface $M$. 
This is the total space of a fibration 
$$
\P(E) \rightarrow \Sigma
$$
where $E$ is a holomorphic bundle of rank $2$ on a compact Riemann surface $\Sigma$.
In that case, the existence of extremal metrics is related to the stability of the bundle $E$.
A lot of work has been done in this direction, we refer to \cite{acgt} for a survey on this topic.

Moreover, in this paper, Apostolov, Calderbank, Gauduchon and T\o{}nnesen-Friedman prove that 
if the genus of $\Sigma$ is greater than two, then $M$ admits a metric of constant scalar curvature in some class 
if and only if $E$ is polystable. 
Another result due to T\o{}nnesen-Friedman \cite{tf} is that if the genus of $\Sigma$
is greater than two, then there exists an extremal K\"ahler metric of non-constant scalar curvature on $M$
if and only if $M=\P(\mathcal{O} \oplus L)$ with $L$ a line bundle of positive degree (see also \cite{s2}).
Note that in that case the bundle is unstable.

The above results admit partial counterparts in the case of parabolic ruled surfaces (see definition \ref{parab}).
In the papers \cite{rs} and \cite{rs2}, Rollin and Singer showed that the parabolic polystability of a parabolic
ruled surface $S$ implies the existence of a constant scalar curvature metric on an iterated blow-up of $S$ encoded by the parabolic structure.

It is natural to ask for such a result in the extremal case. If there exists an extremal metric of non-constant 
scalar curvature on an
iterated blow-up of a parabolic ruled surface, the existence of the extremal vector field implies that
$M$ is of the form $\P(\mathcal{O}\oplus L)$. Moreover, the marked points of the parabolic structure must lie
on the zero or infinity section of the ruling. Inspired by the results mentioned above, we can ask if for every unstable
parabolic structure on a minimal ruled surface of the form $M=\P(\mathcal{O}\oplus L)$, with marked points on the infinity
section of the ruling, 
one can associate an iterated blow-up of $M$ supporting an extremal K\"ahler metric of non-constant scalar 
curvature.

Arezzo, Pacard and Singer, and then Sz\'ekelyhidi, proved that under some stability conditions, 
one can blow-up an extremal K\"ahler manifold and obtain an extremal K\"ahler metric on the blown-up manifold for
sufficiently small metric on the exceptional divisor. This blow-up process enables to prove that many 
of the unstable parabolic structures give rise to extremal K\"ahler metrics of non-constant scalar curvature
on the associated iterated blow-ups. A modification of their argument will enable to get more examples
of extremal metrics on blow-ups encoded by unstable parabolic structures. 
\smallskip

In order to state the result, we need some definitions about parabolic structures.
Let $\Sigma$ be a Riemann surface and $\check{M}$ a geometrically ruled surface, total space of a fibration
$$
\pi:\P(E) \rightarrow \Sigma
$$
with $E$ a holomorphic bundle.

\begin{dfn}
\label{parab}
A \textit{parabolic structure} $\mathcal{P}$ on $$\pi: \check{M}=P(E) \rightarrow \Sigma$$ is the data of $s$ distinct points $(A_i)_{1\leq i \leq s}$ on $\Sigma$ and for each of these points the assignment of a point $B_i\in \pi^{-1}(A_i)$ with a weight $\alpha_i\in(0,1)\cap \Q$.
A geometrically ruled surface endowed with a parabolic structure is called a \textit{parabolic ruled surface}.
\end{dfn}

In the paper \cite{rs}, to each parabolic ruled surface is associated an iterated blow-up 
$$
\Phi: Bl(\check{M},\mathcal{P}) \rightarrow \check{M}.
$$
We will describe the process to construct $Bl(\check{M},\mathcal{P})$ in the case of a parabolic ruled surface whose parabolic structure consists of a single point, the general case being obtained operating the same way for each marked point.
Let $\check{M} \rightarrow \Sigma$ be such a parabolic ruled surface with $A\in\Sigma$, marked point
$Q\in F:= \pi^{-1}(A)$ and weight $\alpha=\dfrac{p}{q}$, 
with $p$ and $q$ coprime integers, $0<p<q$. Denote the expansions of $\dfrac{p}{q}$ and $\dfrac{q-p}{q}$ 
into continuous fractions by:
$$
 \dfrac{p}{q}=\dfrac{1}{e_1-\dfrac{1}{e_2-...\dfrac{1}{e_k}}}
$$
and
$$
\dfrac{q-p}{q}=\dfrac{1}{e'_1-\dfrac{1}{e'_2-...\dfrac{1}{e'_l}}}.
$$
Suppose that the integers $e_i$ and $e_i'$ are greater or equal than two so that these expansions are unique.
Then from \cite{rs} there exists a unique iterated blow-up 
$$
\Phi: Bl(\check{M},\mathcal{P}) \rightarrow \check{M}
$$
with $\Phi^{-1}(F)$ equal to the following chain of curves:

$$
\xymatrix{
{}\ar@{-}[r]^{-e_1} & *+[o][F-]{}
\ar@{-}[r]^{ -e_{2}} &  *+[o][F-]{}
\ar@{--}[r] &  *+[o][F-]{}
\ar@{-}[r]^{-e_{k-1}} &  *+[o][F-]{}
\ar@{-}[r]^{-e_k} &  *+[o][F-]{}
\ar@{-}[r]^{-1} &  *+[o][F-]{}
\ar@{-}[r]^{-e'_{l}} &  *+[o][F-]{}
\ar@{-}[r]^{-e'_{l-1}} &  *+[o][F-]{}
\ar@{--}[r] &  *+[o][F-]{}
\ar@{-}[r]^{-e'_{2}} &  *+[o][F-]{}
\ar@{-}[r]^{ -e'_{1}} & 
}
$$

The edges stand for the rational curves, with self-intersection number above them. 
The dots are the intersection of the curves, of intersection number $1$.
Moreover, the curve of self-intersection $-e_1$ is the proper transform of the fiber $F$. In order to get this blow-up, start by blowing-up the marked point and obtain the following curves:
$$
\xymatrix{
{}\ar@{-}[r]^{-1} & *+[o][F-]{}
\ar@{-}[r]^{-1} & 
}.
$$
Here the curve on the left is the proper transform of the fiber and the one on the right is the first exceptional divisor.
Then blow-up the intersection of these two curves to obtain
$$
\xymatrix{
{}\ar@{-}[r]^{-2} & *+[o][F-]{}
\ar@{-}[r]^{-1} &  *+[o][F-]{}
\ar@{-}[r]^{-2} & 
}.
$$
Then choosing one of the two intersection points that the last exceptional divisor 
gives and iterating the process, one obtain the following chain of curves

$$
\xymatrix{
{}\ar@{-}[r]^{-e_1} & *+[o][F-]{}
\ar@{-}[r]^{ -e_{2}} &  *+[o][F-]{}
\ar@{--}[r] &  *+[o][F-]{}
\ar@{-}[r]^{-e_{k-1}} &  *+[o][F-]{}
\ar@{-}[r]^{-e_k} &  *+[o][F-]{}
\ar@{-}[r]^{-1} &  *+[o][F-]{}
\ar@{-}[r]^{-e'_{l}} &  *+[o][F-]{}
\ar@{-}[r]^{-e'_{l-1}} &  *+[o][F-]{}
\ar@{--}[r] &  *+[o][F-]{}
\ar@{-}[r]^{-e'_{2}} &  *+[o][F-]{}
\ar@{-}[r]^{ -e'_{1}} & 
}
$$

\begin{rmk}
The chain of curves on the left of the one of self-intersection number $-1$ 
is the chain of a minimal resolution of a singularity of $A_{p,q}$ type and the one on the right of a singularity of 
$A_{q-p,q}$ type (see section~\ref{secHJsing}).
\end{rmk}

\begin{rmk}
In \cite{rs}, the curve of self-intersection $-e_1$ is the proper transform of the exceptional divisor of the first blow-up while here this is the proper transform of the fiber $F$.
\end{rmk}

Recall that the zero section of a ruled surface $\P(\mathcal{O} \oplus L)$ is the section given by the zero section of 
$L \rightarrow \Sigma$ and the inclusion $L \subset \P(\mathcal{O} \oplus L)$. The infinity section is given by the zero section
of $\mathcal{O} \rightarrow \Sigma$ in the inclusion $\mathcal{O} \subset \P(\mathcal{O} \oplus L)$.
Given a surface $\Sigma$, $\mathcal{K}$ stands for its canonical bundle and if $A\in\Sigma$, $[A]$ is the line bundle associated to the divisor $A$.
Then we can state:

\begin{theointro}
 \label{theoB}
Let $r$ and $(q_j)_{j=1..s}$ be positive integers such that for each $j$,
 $q_j\geq 3$ and 
$$gcd(q_j,r)=1.$$
 For each $j$, let 
$$p_j\, \equiv \, -r \, [q_j], \: 0<p_j<q_j\, ,\, n_j=\dfrac{p_j+r}{q_j}.$$
Let $\Sigma$ be a Riemann surface of genus $g$ and $s$ marked points $(A_j)$ on it. 
Define a parabolic structure $\mathcal{P}$ on
$$
\check{M}=\P(\mathcal{O} \oplus (\mathcal{K}^r\otimes_j[A_j]^{r-n_j}))
$$
consisting of the points $(B_j)$ in the infinity section of the ruling of $\check{M}$ over the points 
$(A_j)$ together with the weights $(\frac{p_j}{q_j})$. 
If 
$$
\chi(\Sigma)-\sum_j (1-\dfrac{1}{q_j})<0
$$
then there exists an extremal K\"ahler metric of non-constant scalar curvature on $Bl(\check{M},\mathcal{P})$.
This metric is not small on every exceptional divisor.
\end{theointro}

\begin{rmk}
 The parabolic structure is unstable. We will see that the infinity section destabilises 
the parabolic surface.
\end{rmk}

\begin{rmk}
 The K\"ahler classes of the blow-up which admits the extremal metric can be explicitly computed; 
this will be explained in Section~\ref{classes}.
Moreover, these classes are different from the one that could be obtained from the work of Arezzo, Pacard and Singer.
\end{rmk}

Using a slightly more general construction, we will obtain:

\begin{theointro}
 \label{general}
Let $M=\P(\mathcal{O}\oplus L)$ be a ruled surface over a Riemann surface of genus $g$, with $L$ a line bundle of degree $d$. 
If $g\geq 2$ we suppose $d=2g-2$ or $d\geq 4g-3$. 
Then there exists explicit unstable parabolic structures on $M$ such that each associated
iterated blow-up $Bl(M,\mathcal{P})$ admits an extremal K\"ahler metric of non-constant scalar curvature.
The K\"ahler class obtained is not small on every exceptional divisor.
\end{theointro}

\begin{rmk}
 In fact, a combination of the results in \cite{s1} and \cite{ss} shows that the extremal K\"ahler metrics obtained by T\o{}nnesen-Friedman in \cite{tf}
lie exactly in the K\"ahler classes that give relatively stable polarizations. Thus the unstable parabolic structures obtained might be in fact
''relatively stable`` parabolic structures, in a sense that remains to be understood.
\end{rmk}

\subsection{Example}

Consider $\widetilde{\C\P^1\times\T^2}$ a three times iterated blow-up of the total space of the fibration
$$
\C\P^1\times\T^2 \rightarrow \T^2.
$$
The considered blow-up contains the following chain of curves:
$$
\xymatrix{
{}\ar@{-}[r]^{E_1}_{-2} & *+[o][F-]{}
\ar@{-}[r]^{E_2}_{-2} &  *+[o][F-]{}
\ar@{-}[r]^{E_3}_{-1} &  *+[o][F-]{}
\ar@{-}[r]^{ F}_{-3} & 
}
$$
Here $F$ is the proper transform of a fiber of the ruling 
$$
\C\P^1 \times \T ^2 \rightarrow \T^2 .
$$
$E_1$ , $E_2$ and $E_3$ are the exceptional divisors of the iterated blow-ups. Let $S_0$ be the 
proper transform of the zero section. Then :
\begin{theointro}
\label{cp1 t2}
For each $(a,b)\in\R_+^{*,2}$ such that $\frac{a}{b}<k_2$ where $k_2$
is a constant defined in \cite{tf} and for each $(a_1,a_2,a_3)$ positive numbers, there exists $\ep_0>0$
such that for every $\ep\in(0,\ep_0)$,
there exists an extremal metric $\om$ with non-constant scalar curvature on $\widetilde{\C\P^1\times\T^2}$.
This metric satisfies
$$
[\om]\cdot S_0=\frac{2}{3}\pi b,
$$
$$
[\om]\cdot F=\ep^2 a_1, \; [\om]\cdot E_1=\ep^2 a_2, \; [\om]\cdot  E_2=\ep^2 a_3
$$
and
$$
[\om]\cdot E_3=(b-a)\frac{\pi}{3}.
$$
\end{theointro}

\subsection{Strategy}
The Theorem \ref{theoB} will be obtained from a general process. The first step is to consider K\"ahler orbifolds endowed with extremal metrics. 
Such orbifolds can be obtained from the work of Bryant \cite{b} and Abreu \cite{a} on weighted projective spaces. 
Legendre also provide examples in the toric case \cite{le}.
Other examples will come from the work of T\o{}nnesen-Friedman \cite{tf}, generalized to the orbifold setting. These orbifolds will have isolated Hirzebruch-Jung singularities.
The work of Joyce and then Calderbank and Singer enables us to endow a local model of resolution of these singularities with a scalar-flat K\"ahler metric \cite{cs}. 
Then the gluing method of Arezzo, Pacard and Singer \cite{aps} is used to glue these models to the orbifolds and obtain manifolds with extremal K\"ahler metrics.
Note that there exists an improvement of the arguments of \cite{aps} by Sz\'ekelyhidi in \cite{s}.

\begin{rmk}
The gluing process described in \cite {aps} works in higher dimension but there are no 
such metrics on local models of every resolution of isolated singularities in higher dimension. 
However, Joyce constructed ALE scalar-flat metrics on Crepant resolutions \cite{j}. 
Then one can expect to generalize the process described below in some cases of higher dimension.
\end{rmk}

\subsection{Second example}
Using this gluing method we obtain an other simple example.
Let us consider $\widetilde{\C\P^2} $ the three-times iterated blow-up of $\C\P^2$ with the following chain of curves :
$$\xymatrix{
{}\ar@{-}[r]^{H}_{-2} & *+[o][F-]{}
\ar@{-}[r]^{E_3}_{-1} &  *+[o][F-]{}
\ar@{-}[r]^{E_2}_{-2} &  *+[o][F-]{}
\ar@{-}[r]^{ E_1}_{-2} & 
} $$

Here $H$ denotes the proper transform of a line in $\C\P^2$ on which the first blown-up point lies. 
$E_1$, $E_2$ and $E_3$ stand for the proper transform of the first, second and last exceptional divisors. The dots represent the intersections and the numbers below the lines are the self-intersection numbers. Then we can state:

\begin{theointro}
\label{cp2}
For every $a,a_1,a_2,a_3$ positive numbers there exists $\ep_0>0$ such that for every $\ep\in(0,\ep_0)$, 
there is an extremal K\"ahler metric $\om_{\ep}$ of non-constant scalar curvature on $\widetilde{\C\P^2} $ satisfying
$$
[\om_{\ep}]\cdot H=\ep^2 a_3,
$$
$$
[\om_{\ep}]\cdot E_3=a, \; [\om_{\ep}]\cdot E_2=\ep^2 a_2
$$
and
$$
[\om_{\ep}]\cdot E_1=\ep^2 a_1.
$$
\end{theointro}

\begin{rmk} If one starts with the first Hirzebruch surface endowed with the Calabi metric and use the work of
Arezzo, Pacard and Singer to construct extremal metrics on $\widetilde{\C\P^2}$, 
the K\"ahler classes obtained are of the form
$$
[\om_{\ep}]\cdot H=b,
$$
$$
[\om_{\ep}]\cdot E_3=\ep^2 a_3, \; [\om_{\ep}]\cdot E_2=\ep^2 a_2
$$
and
$$
[\om_{\ep}]\cdot E_1=a.
$$
with $a$ and $b$ positive real numbers and $\ep$ small enough .
The K\"ahler classes obtained with the new process can be chosen arbitrarily far from the one obtained 
by Arezzo, Pacard and Singer.
\end{rmk}

\subsection{Plan of the paper}
In the section~\ref{secHJsing} we set up the general gluing theorem for resolutions following \cite{ap1} and \cite{aps}.
In section~\ref{surfaces2} we build the orbifolds with extremal metrics that we use in the gluing construction, and identify the surfaces obtained after resolution. This will prove Theorem~\ref{theoB}.
Then in the section~\ref{theo D} we discuss unstable parabolic structures and give the proof of the Theorem~\ref{general}.
In the last section we show how to obtain the examples of the introduction.

\subsection*{Acknowledgments} 
I'd like to thank especially my advisor Yann Rollin for his help and encouragement. 
I am grateful to Paul Gauduchon and Michael Singer for all the discussions we had. 
I'd also like to thank Vestislav Apostolov who pointed to me Abreu's and Legendre's work. 
I thank Frank Pacard and Gabor Sz\'ekelyhidi for their remarks on the first version of this paper, as well as the referee whose comments enabled to improve the paper.
And last but not least a special thank to Andrew Clarke for all the time he spend listening to me and all the suggestions
he made to improve this paper.

\section{Hirzebruch-Jung singularities and extremal metrics}
\label{secHJsing}
The aim of this section is to present the method of desingularization of extremal K\"ahler orbifolds.

\subsection{Local model}
\label{locmod}
We first present the local model which is used to resolve the singularities.

\begin{dfn} 
\label{HJsing}
Let $p$ and $q$ be coprime non-zero integers, with $p<q$. Define the group $\Gamma_{p,q}$ to be the 
multiplicative subgroup of $U(2)$ generated by the matrix 
$$
\gamma := \left ( \begin{array} {cc}
 \exp \left ( \dfrac{2i \pi}{q} \right )& 0 \\
 0 & \exp \left (\dfrac{2i \pi p}{q} \right )
\end{array} \right )
$$
\end{dfn}
The group $\Gamma_{p,q}$ acts on $\C^2$ :
$$
 \forall (z_0,z_1)\in \C^2 , \gamma .(z_0,z_1):= \left ( \exp \left ( \dfrac{2i \pi}{q} \right ) \cdot z_0,\exp \left ( \dfrac{2i \pi p}{q} \right ) \cdot z_1 \right ).
$$
\begin{dfn}
Let $p$ and $q$ be coprime non zero integers, with $p<q$. An $A_ {p,q}$ singularity is a singularity isomorphic to $\C^2 / \Gamma_{p,q}$. A \textit{singularity of Hirzebruch-Jung type} is any singularity of this type.
\end{dfn}
 
We recall some results about the resolutions of these singularities. First, from the algebraic point of view, $\C^2/ \Gamma_{p,q}$ is a complex orbifold with an isolated singularity at $0$. There exists a minimal resolution 
$$
 \pi : Y_{p,q} \rightarrow \C^2 / \Gamma_{p,q}
$$
called the Hirzebruch-Jung resolution. The manifold $Y_{p,q}$ is a complex surface with exceptional divisor $E:= \pi ^{-1}(0)$ and $\pi$ is a biholomorphism from $Y_{p,q} - E $ to $\C^2-\{0\}/\Gamma$. For more details about resolutions see \cite{bpv}.
Next, $\C^2/ \Gamma_{p,q}$ and $Y_{p,q}$ are toric manifolds. The action of the torus $\T^2$ is the one that comes from the diagonal action on $\C^2$. For more details on this aspect of the resolution see \cite{f}.
Lastly, the minimal resolution can be endowed with an ALE scalar-flat K\"ahler metric $\om_r$ in each K\"ahler class as constructed by Joyce, Calderbank and Singer in \cite{cs}. 
The exceptional divisor of the resolution is the union of $\C\P^1$s embedded in $Y_{p,q}$ and the volume of each of these curves can be chosen arbitrarily.
This metric is $\T^2$-invariant and its behaviour at infinity is controlled:

\begin{prop}
 \textbf{(\cite{rs}, Corollary 6.4.2.)}
In the holomorphic chart 
$$\C^2-\{0\} / \Gamma_{p,q}$$ 
the metric $\om_r$ is given by $\om_r=dd^cf$, with 
$$
f(z)=\frac{1}{2}\vert z \vert^2 + a \, log(\vert z \vert ^2) + \mathcal{O} (\vert z \vert ^{-1})
$$
and $a\leq 0$.
\end{prop}

\subsection{The gluing method}
The gluing method presented here comes from \cite{aps}.
Let $(M,J,\om)$ be a K\"ahler orbifold with extremal metric. Suppose that the singular points of $M$ are isolated and of Hirzebruch-Jung type.
Denote by $p_i$ the singular points of $M$ and $B(p_i,\ep):=B(p_i,\ep)/\Gamma_i$ orbifold balls around the singularities of radius $\ep$ with respect to the metric of $M$.
Fix $r_0>0$ such that the $B(p_i,\ep)$ are disjoint for $\ep<r_0$. 
Consider, for $0<\ep< r_0$, the manifold $M_{\ep}:=M-\cup B(p_i,\ep)$ .
Let $Y_i$ stand for a local model of the resolution of the singularity $p_i$, endowed with the metric of 
Joyce-Calderbank-Singer.
The aim is to glue the $Y_i$ to $M_{\ep}$ in order to obtain a smooth K\"ahler manifold $\widetilde{M}$ which resolves $M$ and has an extremal K\"ahler metric. To do this, one needs to perturb the K\"ahler potentials of the metrics to make them agree on the boundaries of the different pieces, keeping the extremal condition on these potentials. If we consider small enough $\ep$, the metric will look like the euclidian metric in holomorphic chart because it is K\"ahler. On the other hand, the Joyce-Calderbank-Singer metric is ALE so one can hope to glue the metrics together with a slight perturbation. 

Let $s$ be the scalar curvature of $\om$. 
Define the operator :
$$
\begin{array}{llll}
 P_{\om} : & C^{\infty}(M) & \rightarrow & \Lambda^{0,1}(M,T^{1,0}) \\
  & f & \mapsto  &\frac{1}{2} \delb \Xi f
\end{array}
$$
with 
$$
\Xi f = J\nabla f +i \nabla f.
$$
A result of Calabi asserts that a metric is extremal if and only if the gradient field of the scalar curvature is a real holomorphic vector field.
Therefore a metric $\om'$ is extremal if and only if $P_{\om'}(s(\om'))=0$, with $s(\om')$ denoting the scalar curvature.
Let $P_{\om}^*$ be the adjoint operator of $P_{\om}$.
We will use the following proposition:

\begin{prop} \textbf{\cite{li}} 
\label{li}
$\Xi \in T^{1,0} $ is a Killing vector field with zeros if and only if there is a real function 
$f$ solution of $P_{\om}^*P_{\om}(f)=0$ such that $\om(\Xi,.)=-df$.
\end{prop}

This result is initially proved for manifolds but the proof extends directly to orbifolds with isolated singularities, working equivariently in the orbifold charts.

A result of Calabi (\cite{c2}) states that the isometry group of an extremal metric is 
a maximal compact subgroup of the group of biholomorphisms of the manifold. 
Thus in the gluing process we can prescribe a compact subgroup $T$ of the group of biholomorphisms of $M$ 
to become a subgroup of the isometry group of $\widetilde{M}$ and work 
$T$-equivariantly.
We want this group $T$ to be contained in the isometry group of $M$ because the metric that will be obtained on 
$\widetilde{M}$ will be near to the one on $M$ away from the exceptional divisors. Moreover, its algebra must contain the 
extremal vector field of $\om$ for the same reason.
Let $K$ be the sugroup of $Isom(M,\om)$ consisting of exact symplectomorphisms.
Let $T$ be a compact subgroup of $K$ such that its Lie algebra $\mathfrak{t}$ contains $X_s:=J\nabla s$, the extremal 
vector field of the metric.
Let $\h$ be the Lie algebra of real-holomorphic hamiltonian vector fields which are $T$-invariant.
These are the vector fields that remains in the $T$-equivariant setup.
$X_s$ is contained in $\h$. $\h$ splits as $\h' \oplus \h''$ with $\h'=\h \cap \mathfrak{t}$.
The deformations of the metric must preserve the extremal condition so we consider deformations 
$$
 f\longmapsto \om_e + i\del \delb f
$$
such that
\begin{equation}
\label{extremcond}
-ds(\om+ i \del \delb f) = (\om +i \del \delb f) ( X_s+ Y, \: .\: ) \quad 
\end{equation}
with $Y \in \h'$.
As $X_s + Y \in \h'$, the proposition~\ref{li} of Lichnerowicz above ensures that these deformations are extremal.
Moreover, the vector fields from $\h'$ are precisely the ones that give extremal deformations.

In order to obtain such deformations, consider the moment map $\xi_{\om}$ associated to the action of $K$:
$$
\xi_{\om} : M \rightarrow \k^*.
$$
$\xi_{\om}$ is defined such that for every $X \in \k$ the function $\langle \xi_{\om},X\rangle$ on $M$ is a 
hamiltonian for $X$, wich means 
$$
\om(X,\: . \:)= - d\langle \xi_{\om},X\rangle.
$$
Moreover, $\xi_{\om}$ is normalized such that
$$
\int_M \langle \xi_{\om},X\rangle \,\om^n =0.
$$
The equation (\ref{extremcond}) can now be reformulated:
$$
s(\om+i\del \delb f)= \langle \xi_{\om + i \del \delb f} , X_s+Y\rangle + constant .
$$ 
If we work $T$-equivariantly, we are interested in the operator :
$$
\left .
\begin{array}{llll}
 F: &  \h \times C^{\infty}(M)^T \times \R & \rightarrow & C^{\infty}(M)^T \\
  & (X,f,c) & \mapsto  &s(\om+i\del \delb f)- \langle \xi_{\om + i \del \delb f} , X \rangle -c-c_s
\end{array}
\right .
$$
where $C^{\infty}(M)^T$ stands for the $T$-invariant functions and $c_s$ is the average of the scalar curvature of $\om$.
There is a result which is due to Calabi and Lebrun-Simanca:

\begin{prop}\textbf{\cite{aps}} If $\om$ is extremal and if $X_s \in \h$, then the linearization of $F$ at $0$ is given by 
$$(f,X,c) \mapsto-\dfrac{1}{2}P_{\om}^*P_{\om}f- \langle \xi_{\om } , X \rangle -c.$$
\end{prop}

The vector fields in $\h'$ are the ''good ones" as they will give perturbations in the isometry group of the future metric. 
It remains to use the implicit function theorem to get the scalar curvature in term of these vector fields and the algebra $\h''$ stands for the obstruction. 

\subsection{Extremal metrics on resolutions}
\label{perturb}

We choose a group $T$ of isometries of $M$ so that working $T$-equivariantly will simplify the analysis. 
It is necessary to choose a neighborhood of the singularities in which $T$ will appear as a subgroup of the 
isometry group of the metric of Joyce-Calderbank-Singer. 
Moreover, in order to lift the action of $T$ to $\widetilde{M}$, 
it is necessary that $T$ fixes the singularities. 
We will see in lemma \ref{orbicontext} that if we let $T$ be a maximal torus in $K$ then these conditions will be satisfied.
Thus the equivariant setup of \cite{aps} can be used all the same in 
this orbifold case. Lastly, we will have $\h = \t $ and no obstruction will appear in the analysis. 
Following \cite{ap1} and \cite{aps}, we can state the theorem:

\begin{theo}
\label{generalresult}
Let $(M,\om)$ be an extremal K\"ahler orbifold of dimension $2$. 
Suppose that the singularities of $M$ are isolated and of Hirzebruch-Jung type.
Denote by $\pi: \widetilde{M} \rightarrow M$ the minimal resolution of $M$ obtained using the Hirzebruch-Jung strings.
Denote by $E_j$ the $\C\P^1$s that forms the Hirzebruch-Jung strings in the resolution.
Then for every choice of positive numbers $a_j$ 
there exists $\ep_0>0$ such that $\forall \ep \in (0,\ep_0)$ there is an extremal K\"ahler metric on 
$\widetilde{M}$ in the K\"ahler class 
$$
[\pi^*\om] - \ep^2 \sum_j a_j PD[E_j]
$$
\end{theo}

\begin{rmk}
We notice that if $s(\om)$ is not constant, the metrics obtained on the resolution are extremal 
of non-constant scalar curvature. 
Indeed, the metrics converge to $\pi^*\om$ away from the exceptional divisors, so does the scalar curvature.
On the other hand, if $\om$ is of constant scalar curvature, then the extremal metric obtained on the
resolution need not be of constant scalar curvature. Genericity and balancing conditions have to be satisfied
to preserve a CSC metric \cite{ap2} and \cite{rs2}.
\end{rmk}

\begin{rmk}
The proof is the one in \cite{aps} so we refer to this text. 
The tools and ideas are used here in an orbifold context, using the work of \cite{ap1}.
In the paper \cite{ap1} the gluing method is developped in the orbifold context for constant scalar curvature metrics.
On the other hand, the paper \cite{aps} deals with extremal metrics but in the smooth case.
One of the differences in the analysis between the constant scalar curvature case and the extremal case is that one needs
to lift objects to the resolution such as holomorphic vector fields.
Thus we will only give a proof of the following lemma which ensures that we can lift the vector fields needed 
during the analysis.
\end{rmk}

\begin{lemma}
\label{orbicontext}
Let $(M,\om)$ be an extremal K\"ahler orbifold of dimension $2$. 
Suppose that the singularities of $M$ are isolated and of Hirzebruch-Jung type. 
Let $K$ be the subgroup of $Isom(M,\om)$ which are exact symplectomorphisms.
Let $T$ be a maxixmal torus in $K$. 
Then $T$ fixes the singularities.
Its Lie algebra contains the extremal vector field.
Moreover at each singularity of $M$ with orbifold group $\Gamma$ there exists an orbifold chart $U/\Gamma$, $U\subset \C^2$
such that 
in this chart $T$ appears as a subgroup of the torus
acting in the standard way on $\C^2$.
\end{lemma}

\begin{proof}
Let $p\in M$ be a singularity with orbifold group $\Gamma$ and $U\subset \C^2$ such that 
$U/\Gamma$ is an open neighbourhood of $p$ in $M$. $T\subset Isom(M,\om)$ so we can lift the action of $T$
to $U$ such that it commutes with the action of $\Gamma$. Thus $T$ fixes $p$.\\
We see that $X_s$ belongs to $\t$. Indeed, 
$$
\forall X\in\h, [X,X_s]=\mathcal{L}_X X_s=0
$$
because $X$ preserves the metric. Thus $X_s\in \h$ and as $\h=\t$, $X_s \in \t$.
Then we follow the proof of \cite{aps}. By a result of Cartan, we can find holomorphic coordinates on $U$ such that 
the action of $T$ is linear. More than that, we can suppose that the lift of $\om$ in these coordinates $(z_1,z_2)$
satisfies
$$
\om=\del \delb (\dfrac{1}{2} \vert z \vert^2 +\phi)
$$
where $\phi$ is $T$ invariant and $\phi=O(\vert z \vert^4)$.
In this coordinates $T$ appears as a subgroup of $U(2)$.
Thus $T$ is conjugate to a group whose action on $\C^2$ is diagonal.
As $T$ and $\Gamma$ commute, we can diagonalize simultaneously these groups.
In the new coordinates, the action of $\Gamma$ is described by definition \ref{HJsing} and $T\subset \T^2$
with $\T^2$ action on $\C^2$ given by
$$
\begin{array}{ccc}
 \T^2 \times \C^2 & \rightarrow & \C^2 \\
 (\theta_1,\theta_2),(w_1,w_2) & \mapsto & (e^{i\theta_1}w_1,e^{i\theta_2} w_2) \\
\end{array}
$$
\end{proof}

\begin{rmk}
Note that there exists a refinement of the proof of Arezzo, Pacard and Singer in the blow-up case. 
This is given in the paper of Sz\'ekelyhidi \cite{s}.
Basically, the idea is to glue the metric of Burns-Simanca to the extremal metric on the blow-up manifold.
This give an almost extremal metric. It remains to perturb the metric to obtain an extremal one. 
Moving the blown-up point if necessary, this problem becomes a finite dimensional problem.
This argument might be used in the case of resolution of isolated singularities. 
\end{rmk}

\section{Extremal metrics on orbisurfaces}
\label{surfaces2}

The aim of this section is to prove Theorem~\ref{theoB}. We first construct extremal metrics on special orbisurfaces. Then we apply Theorem~\ref{generalresult} to these examples, and identify the smooth surfaces obtained after resolution.

\subsection{Construction of the orbifolds}
\label{construct2}
In this section we generalize to an orbifold setting the so-called Calabi construction. We shall construct extremal metrics on projectivization of rank $2$ orbibundles over orbifold Riemann surfaces.
We will focus on the case where the orbifold Euler characteristic $\chi^{orb}$ of the Riemann surface is strictly negative, thought the method would extend directly to the other cases. See for example \cite{gbook} for a unified treatment of the construction in the smooth case. Our restriction will present the advantage of being very explicit and will help to keep track of the manifolds considered in section~\ref{surfaces}.

The starting point is the pseudo-Hirzebruch surfaces constructed by T\o{}nnesen-Friedman in \cite{tf}.
These are total spaces of fibrations
$$
\P(\mathcal{O} \oplus L) \rightarrow \Sigma_g
$$
where $\Sigma_g$ is a Riemann surface of genus $g$ and $L$ a positive line bundle. We briefly recall the construction of extremal metrics on such a surface.
Let 
$$
U:=\lbrace (z_0,z_1) / \vert z_0 \vert ^2 > \vert z_1 \vert ^2 \rbrace \subset \C^2
$$
and 
$$
\D:= \lbrace (z_0,z_1) \in U / z_0=1 \rbrace.
$$
We can then consider $U$ as a principal bundle over the Poincar\'e disc $\D$ which admits the trivialisation :
$$
\begin{array}{ccc}
 U & \rightarrow & \C^*\times \D \\
 (z_0,z_1) & \mapsto & (z_0, (1,\dfrac{z_1}{z_0})).
\end{array}
$$
The vector bundle associated to $U$ is trivial and we can consider it as the extension of $U$ over zero. We will denote by $U^q$ the tensor powers of this vector bundle.

Recall that $U(1,1)$ is the group of isomorphisms of $\C^2$ which preserve the form 
$$
u(z,w)= - z_0\overline{w_0}+z_1\overline{w_1}.
$$
Moreover,
$$
U(1,1)=\lbrace e^{i\theta} \left ( \begin {array} {cc} 
 \alpha & \overline{\beta} \\
  \beta & \overline{\alpha} 
\end{array} \right ) ; \alpha \overline{\alpha} - \beta \overline{\beta} =1; \theta \in \R, (\alpha, \beta) \in \C^2 \rbrace .
$$
and $U(1,1)$ acts on $U$.
One of the central results of this work of T\o{}nnesen-Friedman is then (see \cite{tf}):
\begin{theo}
\label{t}
Let $q$ be a non zero integer. There exists a constant $k_q$ such that for every choice of constants $0<a<b$ such that $\frac{b}{a}<k_q$, there is a $U(1,1)$-invariant extremal K\"ahler metric with non-constant scalar curvature on $U$.
This metric can be extended in a smooth way on $\P(\mathcal{O} \oplus U^q)$.
\end{theo}

\begin{rmk} 
The value of the scalar curvature is $-\frac{4}{b}$ on the zero section and $-\frac{4}{a}$ on the infinity section.
\end{rmk}

\begin{rmk}
Here the considered manifolds are non-compact and we say that the metric is extremal if the vector field $J\nabla s$ associated to the metric is holomorphic.
\end{rmk}

T\o{}nnesen-Friedman uses this result to obtain extremal metrics on pseudo-Hirzebruch surfaces. Let $\Sigma$ a Riemann surface of genus greater or equal than two and let $\Gamma$ be its fundamental group. $\Gamma$ can be represented as a subgroup of $U(1,1)$ and we obtain a holomorphic bundle:
$$
\P(\mathcal{O} \oplus U^q) / \Gamma \rightarrow \D / \Gamma= \Sigma .
$$
T\o{}nnesen-Friedman shows that $\P(\mathcal{O} \oplus U^q)$ is isomorphic to $\P(\mathcal{O} \oplus \mathcal{K}^{\frac{q}{2}})$ 
where $\mathcal{K}$ stands for the canonical line bundle.
The result of T\o{}nnesen-Friedman provides extremal metrics on $\P(\mathcal{O} \oplus  \mathcal{K}^{\frac{q}{2}})$.

Now we consider an orbifold Riemann surface $\overline{\Sigma}$ of genus $g$, and refer to \cite{rs} for more details on orbifold Riemann surfaces.
The genus $g$ is no longer assumed to be greater than two but we assume that the orbifold Euler characteristic is 
strictly negative. In that case, the orbifold fundamental group $\Gamma$ of $\overline{\Sigma}$ can be represented as 
a subgroup of $U(1,1)$ such that
$$
\overline{\Sigma}=\D/\Gamma.
$$ 
We saw that 
$$
U\cong \C^* \times \D
$$
The action of $U(1,1)$ in this chart is given by:
$$
\gamma\cdot (\xi, z) \mapsto ( e^{i \theta} (\alpha + \overline{\beta} z) \xi , \dfrac{\beta + \overline{\alpha}z}{\alpha + \overline{\beta}z})
$$
where
$$
\gamma = e^{i\theta} \left ( \begin {array} {cc} 
 \alpha & \overline{\beta} \\
  \beta & \overline{\alpha} 
\end{array} \right )\in U(1,1).
$$
The action of $U(1,1)$ on $\C^{\otimes q} \times \D$ is then given by:
$$
\gamma\cdot (\xi, z) \mapsto ( e^{iq \theta} (\alpha + \overline{\beta} z)^q \xi , \dfrac{\beta + \overline{\alpha}z}{\alpha + \overline{\beta}z})
$$
and the change of coordinates 
$$
(\xi, z) \rightarrow (\xi^{-1},z)
$$
enables to extend this action and to define 
$$
\overline{M}=\P(\mathcal{O} \oplus U ^q)/ \Gamma .
$$
This naturally fibres over $\overline{\Sigma}$ and define an orbifold bundle
$$
\pi:\overline{M} \rightarrow \overline{\Sigma }.
$$
We adopt the convention from \cite{rt} for the definition below.
\begin{dfn}
 An orbifold line bundle over an orbifold $M$ is given by local invariant line bundles $L_i$ over each orbifold charts $U_i$ such that the following cocycle condition is satisfied:\\
Suppose that $V_1$, $V_2$ and $V_3$ are open sets in $M$ with orbifold groups $G_i$, and orbifold charts 
$U_i$, such that $V_i=U_i/G_i$, $i=1..3$.

Then by definition of an orbifold there are charts $U_{ij}$ such that 
$V_i\cap V_j=V_{ij} \cong U_{ij}/G_{ij}$ with inclusions 
$U_{ij} \rightarrow U_i$ and $G_{ij} \rightarrow G_i$, $\lbrace i,j\rbrace \subset \lbrace 1,2,3 \rbrace $.
Pulling back $L_j$ and $L_i$ to $U_{ij}$, there exists an isomorphism $\phi_{ij}$ from $L_j$ to $L_i$ intertwining the actions of $G_{ij}$. 
Moreover, pulling-back to $U_{123}$, the cocycle condition is that over $U_{123}$ we have:
$$
\phi_{12} \phi_{23} \phi_{31}=1 \in L_1\otimes L_2^*\otimes L_2\otimes L_3^*\otimes L_3\otimes L_1^* .
$$
\end{dfn}
This definition can be generalized to define orbifold vector bundles, tensor products,
direct sums and projectivizations of orbifold bundles.
It is enough to define these operations on orbifold charts $U_i$ and verify that the cocycle condition is still satisfied. The orbifold canonical bundle $\mathcal{K}_{orb}$ is defined to be 
$\mathcal{K}_U$ on each orbifold chart $U$.

With these definitions in mind, we prove the following:
\begin{lemma}
Suppose that $q=2r$.
Then the surface $\overline{M}$ is isomorphic to $\P(\mathcal{O} \oplus \mathcal{K}_{orb}^r)$. 
\end{lemma}
\begin{rmk}
We will see in the study of the singularities that we need to consider $q$ even.
\end{rmk}
\begin{proof}
Following T\o{}nnesen-Friedman (\cite{tf}), we compute the transition functions. Recall that the action of $\Gamma$ is given by:
$$
\forall \gamma \in \Gamma, \, \gamma\cdot (\xi,z)=( e^{2ir \theta} (\alpha + \overline{\beta} z)^{2r} \xi , \dfrac{\beta + \overline{\alpha}z}{\alpha + \overline{\beta}z})
$$
in the chart $\C^{\otimes 2r}\times\D$. As $\D/\Gamma=\overline{\Sigma}$, the transition functions are induced by the maps:
$$
z \mapsto \gamma\cdot z
$$
with $\gamma \in \Gamma$.
Thus the transition functions for the bundle $U^{2r} /\Gamma$ are 
$$
z \mapsto (\alpha + \overline{\beta} z)^{2r}.
$$
Now the transition functions for $\mathcal{K}_{orb}$ are computed by 
$$
d(\gamma\cdot z)= d(\dfrac{\beta + \overline{\alpha}z}{\alpha + \overline{\beta}z})=(\dfrac{1}{\alpha + \overline{\beta}z})^2dz
$$
because $\gamma\in U(1,1)$. It shows that they are equal to
$$
z \mapsto (\alpha + \overline{\beta} z)^{-2}
$$
and $ U^{2r} / \Gamma = \mathcal{K}_{orb}^r$. So $\overline{M}=\P(\mathcal{O} \oplus \mathcal{K}_{orb}^r)$ 
on $\overline{\Sigma}$.

\end{proof}

The extremal metric mentioned in Theorem~\ref{t} is $\Gamma$-invariant as it is $U(1,1)$-invariant 
and the result of \cite{tf} extend to the case of orbibundles:

\begin{prop}
 \label{orbibundles}
Let $q=2r$ be a non zero even integer. There exists a constant $k_q$ such that for every choice of constants $0<a<b$ such that $\frac{b}{a}<k_q$, there is an extremal K\"ahler metric with non-constant scalar curvature on $\P(\mathcal{O} \oplus \mathcal{K}^r_{orb})$.
The restrictions of this metric to the zero and infinity sections of
$$
\P(\mathcal{O} \oplus\mathcal{K}^r_{orb}) \rightarrow \overline{\Sigma}
$$ are constant scalar curvature metrics.
The value of the scalar curvature on the zero section is $-\frac{4}{b}$ and $-\frac{4}{a}$ on the infinity section.
\end{prop}

We have obtained extremal metrics on orbifold ruled surfaces.

\subsection{Singularities and resolution}
\label{sing}
We now proceed to the study of the singularities of the orbifolds.
Let $(A_i)$ denote the singular points of $\overline{\Sigma}$ and let $q_i$  be the order of the singular point $A_i$.

As $\chi^{orb}(\overline{\Sigma})<0$, there exists a morphism :
$$
\overline{\phi}:\Gamma \rightarrow Sl_2(\R)/\Z_2=Isom(\H^2)
$$
such that $\overline{\Sigma}=\H^2/Im(\overline{\phi})$.
The transformation
$$ 
z \mapsto \dfrac{z-i}{z+i}
$$
that sends the half-plane to the Poincar\'e disc gives :
$$ 
\phi_0:\Gamma \rightarrow SU(1,1)/\Z_2 .
$$
We recall a description of $\Gamma$:
$$
\Gamma=< (a_i,b_i)_{i=1..g}\, ,\, (l_i)_{i=1..s} \, \vert \, \Pi [a_i,b_i] \, \Pi l_i\, = l_i^{q_i}=1 >.
$$
Then $\phi_0$ defines matrices $A_i,B_i, L_i$ in $SU(1,1)$ satisfying the relations 
$$ 
\Pi [A_i,B_i]\, \Pi L_i\,=\pm Id \,; \, L_i^{q_i}=\pm Id .
$$
The action of $-Id$ on $\C^{\otimes q}\times \D$ is :
$$ 
(\xi, z) \rightarrow ((-1)^q \xi,z).
$$
In order to obtain isolated singularities, we will suppose that $q$ is even.

To simplify notation we will denote by $\Gamma$ the image of $\phi_0$. The singular points of $\overline{M}$ come from the points of $\P(\mathcal{O} \oplus U^q)$ with non-trivial stabilizer under the action of $\Gamma/\Z_2$. We will deal with points on the zero section in the chart $\C^{\otimes q}\times\D$.
Let $(\xi_0,z_0)$ be a point whose stabilizer under the action of $\Gamma$ is not reduced to $\{\pm Id\}$ and let $\gamma \neq \pm id$ in $\Gamma$ fix $(\xi_0,z_0)$.
The element $\gamma \in U(1,1)$ can be written:
$$
\gamma = \pm\left ( \begin {array} {cc} 
 \alpha & \overline{\beta} \\
  \beta & \overline{\alpha} 
\end{array} \right )
$$
and satisfies
$$
\dfrac{\beta + \overline{\alpha}z_0}{\alpha + \overline{\beta}z_0} =z_0.
$$
The stabilizer of $(\xi_0,z_0)$ is then include in the one of $z_0\in\D$ under the action of $\Gamma$ on $\D$. The point $z_0$ gives a singular point $A_i$ of $\overline{\Sigma}=\D/\Gamma$ and the point $(\xi_0,z_0)$ gives rise to a singular point in $\overline{M}$ in the singular fiber $\pi^{-1}(A_i)$. The isotropy group of $A_i$ is $\Z_{q_i}$ and we can suppose that $\gamma$ is a generator of this group.
As $\gamma$ is an element of $SU(1,1)$ of order $q_i$, its characteristic polynomial is 
$$
X^2\pm 2Re(\alpha)X+1.
$$
If $Re(\alpha)=\delta$ with $\delta \in \{-1,+1\}$, there exists a basis in which $\gamma$ is one of the following matrices:
$$
\left ( \begin {array} {cc} 
 \delta & 0 \\
  0 & \delta 
\end{array} \right ) ;
\left ( \begin {array} {cc} 
 \delta & 1 \\
  0 & \delta
\end{array} \right ) .
$$
It is not possible because $\gamma$ is suppose to be different from $\pm Id$ and of finite order.
Thus the characteristic polynomial of $\gamma$ admits two distinct complex roots and $\gamma$ is diagonalizable in $SU(1,1)$ so we can fix $P\in SU(1,1)$ such that
$$
P^{-1}. \gamma . P =  \left ( \begin {array} {cc} 
 a & 0\\
  0 & b
\end{array} \right ).
$$
Then $\gamma$ is of finite order $q_i$ and $\det(\gamma)=1$ so we can fix $\xi_{q_i}$ a primitive $q_i^{th}$ root of unity such that
$$
P^{-1}. \gamma . P =  \left ( \begin {array} {cc} 
 \xi_{q_i} & 0\\
  0 & \xi_{q_i}^{-1}
\end{array} \right ).
$$
As $P\in SU(1,1)$, $P$ preserves the open set $U$ and induces a change of coordinates in a neighbourhood of the fixed point.\\
Indeed, the action of $P$ on the coordinates $(\xi,z)$ is given by:
$$
P.(\xi,z)=((c + \overline{d} z)^q \xi , \dfrac{d + \overline{c}z}{c+ \overline{d}z})
$$
with
$$
P=\left ( \begin {array} {cc} 
 c & \overline{d} \\
  d & \overline{c} 
\end{array} \right ).
$$
We compute the differential at $(\xi_0,z_0)$:
$$
DP_{(\xi_0,z_0)}=\left ( \begin{array}{cc}
                          q \overline{d} (c+ \overline{d} z_0)^{q-1} \xi_0  & (c+\overline{d} z_0 )^q \\
                          1/(c+\overline{d} z_0)^2 & 0 
                         \end{array} \right ) .
$$
and the determinant of this matrix is
$$
det(DP_{(\xi_0,z_0)})=-(c+\overline{d} z_0)^{q-2}.
$$
Note that $q\geq 2$. As we don't have $c=d=0$, this determinant is zero if and only if $z_0=-\dfrac{c}{\overline{d}}$.
But $\vert z_0 \vert ^2<1$ so it would imply $\vert c\vert ^2 < \vert d \vert ^2$ which is impossible because 
$$
det(P)= \vert c \vert ^2 - \vert d \vert ^2=1.
$$
Thus $P$ defines a change of coordinates near the singular point.
In these new coordinates $(\xi',z')$, the action of $\gamma$ is
$$
(\xi',z') \mapsto (\xi_{q_i} ^{q}\: \xi', {\xi_{q_i}^{-2}}\: z').
$$
The only fixed points of $\gamma$ in this local chart are $(0,0)$ and $(\infty,0)$, 
the point at infinity corresponding to the action
$$
(\xi', z' ) \mapsto (\xi_{q_i} ^{-q}\: \xi', {\xi_{q_i}^{-2}}\: z').
$$
As the singular points of $\overline{\Sigma}$ are isolated, the singular fibers of $\overline{M}$ are isolated.
Moreover, we see that the singular points are on the zero and infinity sections in the initial coordinates system, 
as the transformation $P$ preserves the ruling and the zero and infinity sections.
We can recognize precisely the Hirzebruch-Jung type of these singularities using the method described in \cite{bpv}. To simplify we will suppose $gcd(r,q_j)=1$. If we set $\zeta_i= \xi_{q_i}^{-2}$, the action is
$$
(\xi',z') \mapsto (\zeta_i^{-r}\: \xi', \zeta_i\: z').
$$
Summarizing:
\begin{prop}
\label{orbifold2} Let $\overline{\Sigma}$ be an orbifold Riemann surface with strictly negative orbifold Euler characteristic. Let $(A_i)_{1\leq i \leq s}$ be its singular points that we suppose 
to be of order strictly greater than $2$.
Then if $q=2r$ is an even integer, the orbifold $\P(\mathcal{O} \oplus \mathcal{K}^r_{orb})$ defined in the above construction has $2s$ singular points, two of them in each fiber $\pi^{-1}(A_i)$.
Moreover, if $gcd(r,q_j)=1$, the singular points are of type $A_{p_i,q_i}$ and $A_{q_i-p_i,q_i}$ in each fiber, with $p_i\equiv -r [q_i]$.
\end{prop}

\begin{rmk}
The hypothesis on the order of the singularity of $\overline{\Sigma}$ is 
needed to avoid an isotropy group of the form $\{ \pm Id\}$.
\end{rmk}

We can now apply the Theorem ~\ref{generalresult}. 

\begin{cor}
\label{orbiextrem}
Let $\overline{M}$ be a orbifold surface as in Proposition~\ref{orbifold2}, endowed with an extremal metric arising from Proposition~\ref{orbibundles}.
Let $\pi: \widetilde{M} \rightarrow \overline{M}$ be the Hirzebruch-Jung resolution of $\overline{M}$.
Then, $\widetilde{M}$ admits extremal K\"ahler metrics with non-constant scalar curvature that converge to $\pi^*\om$ on every compact set away from the exceptional divisors.
\end{cor}

\subsection{Identification of the resolution $\widetilde{M}$}
\label{surfaces}
We want to describe the surface that we obtain after desingularization. Let's consider $\overline{M}$ obtained in the Corollary~\ref{orbiextrem}. This is the total space of a 
singular fiber bundle 
$$
\overline{\pi} : \overline{M}=\P(\mathcal{O} \oplus \mathcal{K}^r_{orb}) \rightarrow \overline{\Sigma}.
$$
The $2s$ singular points are situated on $s$ singular fibers $F_i$, each of them admits a singularity of 
type $A_{p_i,q_i}$ on the zero section and of type $A_{q_i-p_i,q_i}$ on the infinity section.

Let $\Sigma$ be the smooth Riemann surface topologically equivalent to $\overline{\Sigma}$. 
We define a parabolic ruled surface $\check{M}$ in the following way: first we set 
$$
n_j=\dfrac{p_j+r}{q_j}
$$
for each singular fiber. Note that by construction $n_j$ is an integer. We define the line bundle $L$ on $\Sigma$:
$$
L= \mathcal{K}^r\otimes_j [A_j]^{r-n_j}
$$
where the $A_j$ stand for the points on $\Sigma$ corresponding to the singular points of $\overline{\Sigma}$. We set $\check{M}$ to be the total space of the fibration
$$
\P(\mathcal{O}\oplus L)\rightarrow \Sigma.
$$
The parabolic structure on $\check{M}$ consists in the $s$ points $A_j\in{\Sigma}$, the points $B_j$ in the fiber over $A_j$ on the 
infinity section and the corresponding weights $\alpha_j:=\dfrac{p_j}{q_j}$. Let $Bl(\check{M},\mathcal{P})$ be the iterated blow-up associated to $\check{M}$ as defined in the introduction.

\begin{prop}
 \label{surface2}
The smooth surface obtained by the minimal resolution of $\overline{M}$ is $Bl(\check{M},\mathcal{P})$.
\end{prop}
Note that together with corollary~\ref{orbiextrem}, proposition~\ref{surface2} ends the proof of Theorem~\ref{theoB}. More precisions on the K\"ahler classes obtained are given in section~\ref{classes}.

\begin{proof}
We denote by 
$$
\overline{\pi}: \widetilde{M} \rightarrow \overline{M}
$$
the minimal resolution of $\overline{M}$. This is a ruled surface so it comes from blow-ups of a minimal ruled surface $M$:
$$
\xymatrix{&          \widetilde{M}\ar[dl]_{\pi} \ar[dr]^{\overline{\pi}}   &  \\
M & & \overline{M} }
$$

$M$ is the total space of a fibration $\P(\mathcal{O} \oplus L')$ over $\Sigma$.
Indeed, $M$ is birationally equivalent to $\overline{M}$ so $M$ is a ruled surface over $\Sigma$. 
As $M$ is minimal, it is of the form $\P(E)$ where $E$ is a holomorphic bundle of rank $2$ on $\Sigma$.
Then, to show that $E$ splits, we consider the vector field $X_s=J\nabla s$ on $\overline{M}$, 
with $s$ the scalar curvature of a metric from Corollary~\ref{orbiextrem}. The vector field $X_s$ is vertical and can be lifted to $\widetilde{M}$. 
It projects to a vertical holomorphic vector field on $M$. The latter admits two zeros on each fiber
because it generates an $S^1$ action and not a $\R$ action. The two zeros of this vector field restricted to each fiber give two holomorphic sections of $\P(E)$ and 
describe the splitting we look for. Tensoring by a line bundle if necessary, we can suppose that 
$\P(E)=\P(\mathcal{O} \oplus L')$.

$$
\xymatrix{&          \widetilde{M}\ar[dl]_{\pi} \ar[dr]^{\overline{\pi}}   &  \\
\P(\mathcal{O} \oplus L') & & \overline{M} }
$$

It remains to recognize $L'$.

We set $M^*:=\D^*\times \C\P^1 / \Gamma$, where $\D^*$ denotes the Poincar\'e disc minus the fixed points under the action 
of $\Gamma$. $M^*$ is the surface obtained from $\overline{M}$ by taking away the singular fibers. As $\overline{\pi}$ and 
$\pi$ are biholomorphisms away from the exceptional divisors, we get a natural injection from $M^*$ to $M$.
$M$ appears as a smooth compactification of $M^*$. Moreover, $M^*=\P(\mathcal{O}\oplus\mathcal{K}^r_{orb})$ on $\Sigma^*$.

$$
\xymatrix{&          \widetilde{M}\ar[dl]_{\pi} \ar[dr]^{\overline{\pi}}   &  \\
\P(\mathcal{O} \oplus L') & & \overline{M}\\
& M^*\ar[ul]\ar[ur]&
}
$$

We can understand the way of going from $\overline{M}$ to $M$ in the neighbourhood of the singular fibers using the 
work of \cite{rs}. If $A_j$ is a singular point of $\overline{\Sigma}$ of order $q_j$ and $\Delta_j$ a small disc 
around $A_j$ then $\Delta_j \times \C\P^1/\Z_{q_j}$ is a neighbourhood of the singular fiber over $A_j$. 
It follows from section~\ref{construct2} that the action of $\Z_{q_j}$ is given by:
$$
\begin{array}{ccc}
\Delta_j \times \C \P^1 & \rightarrow & \Delta_j \times \C \P^1 \\
(z,[u,v]) & \mapsto & (\zeta_{q_j} z, [u,\zeta_{q_j} ^{p_j}v]).
\end{array}
$$
with $\zeta_{q_j}$ a $q_j^{th}$ primitive root of unity and $p_j \equiv -r [q_j]$.
We get a neighbourhood $\Delta_j \times \C\P^1$ of the corresponding point in $M$ by the map:
$$
\begin{array}{cccc}
 \phi_j: & (\Delta_j \times \C \P^1) / \Z^q & \rightarrow & \Delta_j \times \C \P^1\\
  & (z,[u,v]) & \mapsto & (x=z^{q_j}, [u,z^{-p_j} v]).
\end{array}
$$
Indeed, the resolution of Hirzebrugh-Jung singularities and blow-down $-1$-curves are toric processes.
From the theorem 3.3.1. of \cite{rs}, we know that the fan of the Hirzebrugh-Jung resolution of 
$\Delta_j \times \C\P^1/\Z_{q_j}$ is the same as the one of the iterated blow-up $Bl(\Delta_j\times \C\P^1,\mathcal{R})$
of $\Delta_j\times \C\P^1$ described in 
the introduction, with parabolic structure 
$\mathcal{R}$ consisting in the point $(0,[0,1])$ with the weight $\dfrac{p_j}{q_j}$. Moreover, the resulting map from 
$\Delta_j \times \C\P^1/\Z_{q_j}$ to $\Delta_j\times \C\P^1$ is $\phi_j$.
$$
\xymatrix{&          Bl(\Delta_j\times \C\P^1,\mathcal{R})\ar[dl]_{\pi} \ar[dr]^{\overline{\pi}}   &  \\
\Delta_j \times \C\P^1 & &\ar[ll]_{\phi_j} \Delta_j \times \C\P^1/\Z_{q_j}\\
& \Delta_j^* \times \C\P^1\ar[ul]\ar[ur]&
}
$$

We get $M$ by gluing to $M^*$ the open sets $\Delta_j\times \C\P^1$ on the open sets $(\Delta_j^*\times \C\P^1)/ \Z_q$ 
with the maps $\phi_j$. We now show that these maps modify $\P(\mathcal{O}\oplus \mathcal{K}_{orb}^r)$ on 
$\Sigma^*$ to $\P(\mathcal{O}\oplus L)$ on $\Sigma$. 
Indeed, from the identity
$$
dz=x^{\frac{1}{q_j} - 1} dx
$$
on $\Delta_j^*/\Z_{q_j}$, we see that on $\Delta_j^*$, a local trivialization for $L'$ is given by
$$
x \mapsto x^{r-\frac{r}{q_j}} (dx)^{\otimes r}.
$$
Then the maps $\phi_j$ define the coordinates
$$
u'=u, \: v'=z^{-p_j}v=x^{-\frac{p_j}{q_j}}v
$$
wich implies that the trivializing section extends over $\Delta_j$ in $\Sigma$ by
$$
x \mapsto x^{r-\frac{p_j+r}{q_j}} dx^{\otimes r}.
$$
In other words, gluing $\Delta_j\times \C\P^1$ to $M^*$ using the $\phi_j$ is the same as 
tensoring $\mathcal{K}^r_{orb}$ with the $[A_j]^{\frac{-p_j}{q_j}}$.
Thus we conclude that $L'= \mathcal{K}^r\otimes_j [A_j]^{r-n_j}=L$.

So far, we have seen that $\widetilde{M}$ comes from $\P(\mathcal{O}\oplus L)$ from an iterated blow-up. The iterated blow-up
is described locally in the toric framework in \cite{rs} and it corresponds exactly to the one of $Bl(\check{M},\mathcal{P})$,
which ends the proof.
\end{proof}

\subsection{The K\"ahler classes}
\label{classes}
Let $(\overline{M},\om)$ be an orbifold extremal K\"ahler surface and $(\widetilde{M},\om_{\ep})$ 
the smooth surface obtained after desingularization with an extremal metric as in Corollary~\ref{orbiextrem}
$$
\overline{\pi} : \widetilde{M} \rightarrow \overline{M}. 
$$
Away from the exceptional divisors, $\om_{\ep}$ is obtained by a perturbation of the form 
$\om_{\ep}=\om+\del \delb f$ so the K\"ahler class doesn't change. In order to determine $[\om_{\ep}]$ 
it is sufficient to integrate $\om_{\ep}$ along a basis for $H_2(\widetilde{M},\R)$.
We will compute the K\"ahler class in the case of a unique singular fiber, 
the general case can be deduced from this one. Denote by $M=\P(\mathcal{O} \oplus L)$ the minimal model 
associated to $\widetilde{M}$
$$
\pi:\widetilde{M} \rightarrow M.
$$
It's homology $H_2(M,\R)$ is generated by the zero section and the class of a fiber. 
Thus the homology $H_2(\widetilde{M},\R)$ is generated by the proper transform of these 
two cycle and the exceptional divisors. Consider the chain of curves coming from the resolution 
of the singular fiber of $\overline{M}$
$$
\xymatrix{
{}\ar@{-}[r]^{-e_1}_{E_1} & *+[o][F-]{}
\ar@{-}[r]^{ -e_{2}}_{E_2} &  *+[o][F-]{}
\ar@{--}[r] &  *+[o][F-]{}
\ar@{-}[r]^{-e_{k-1}}_{E_{k-1}} &  *+[o][F-]{}
\ar@{-}[r]^{-e_k}_{E_k} &  *+[o][F-]{}
\ar@{-}[r]^{-1}_{S} &  *+[o][F-]{}
\ar@{-}[r]^{-e'_{l}}_{E_l'} &  *+[o][F-]{}
\ar@{-}[r]^{-e'_{l-1}}_{E_{l-1}'} &  *+[o][F-]{}
\ar@{--}[r] &  *+[o][F-]{}
\ar@{-}[r]^{-e'_{2}}_{E_2'} &  *+[o][F-]{}
\ar@{-}[r]^{ -e'_{1}}_{E_1'} & 
} $$

Note that $S$ is the proper transform under $\overline{\pi}$ of the singular fiber of $\overline{M}$. 
$E_1'$ is the proper transform under $\pi$ of a fiber of $M$. The construction of $\om_{\ep}$ shows that integrating 
$$
\sum_j \int_{E_j} \om_{\ep}=\ep^2 V
$$
or
$$
\sum_i \int_{E_i'} \om_{\ep}=\ep^2 V'
$$
for small positive number $\ep$ and $V$ and $V'$ depending on the volume of the metric on each resolution. 
The construction of the metric of Joyce, Calderbank and Singer enables to choose the volume 
of each curve $E_j$ and $E_i'$ provided that the sum
is equal to $V$ and $V'$ respectively.
Thus if we choose $(a_j)$ and $(a_i')$ such that 
$$
\sum_j a_j=V
$$
and
$$
\sum_i a_i'=V'
$$
we get
$$
\int_{E_j} \om_{\ep}=\ep^2 a_j
$$
and
$$
\int_{E_i'} \om_{\ep}=\ep^2 a_i'
$$
If we denote by $S_0$ the proper transform of the zero section of $M$ with $\pi$, 
it remains to compute $[\om_{\ep}]\cdot S_0$ and $[\om_{\ep}]\cdot S$.
These two integrals can be computed on $\overline{M}$ following T\o{}nnesen-Friedman \cite{tf}.
On the zero section, the metric is of constant sectional curvature and 
its scalar curvature is equal to $b$ so the Gauss-Bonnet formula for orbifolds gives
$$
A=\int_{S_0} \om_{\ep} = -\pi b \chi^{orb}.
$$
On $S$, the explicit form of the metric on $U$ enables to compute
$$
B=\int_S \om_{\ep}=2\pi \dfrac{(b-a)}{2rq}.
$$
where $r$ is related to the degree $d$ of $L$ by $d=-r\chi+\sum_j (r-n_j)$, $a$ and $b$ 
are constants satisfying $\dfrac{a}{b}< k_{2r}$ and $q$ is the order of the singularities of the singular fiber.

Now, if we write $\om_{\ep}$ in the basis formed by the Poincar\'e duals of the generators 
$S_0$, $S$, $(E_i)$ and $(E_j')$ of $H_2(\widetilde{M},\R)$ we have
$$
\om_{\ep}=c_0PD(S_0)+c_1PD(E_1)+..+c_kPD(E_k)+cPD(S)
$$
$$
+c_l'PD(E_l')+..+c_1'PD(E_1').
$$
We can compute the vector 
$$
C=[c_0,c_1,..,c_k,c,c_l',..,c_1']^t
$$
using the matrix $Q$ which represents the intersection form of $H_2(\widetilde{M},\R)$ in this basis
$$
Q=\left ( \begin {array} {ccccccccc} 
 -l-1 & 1 & 0 &\cdots & & & & & \\
 1 & -e_1 & 1 &\cdots & & & & & \\
 0 & 1 & -e_2 &\cdots & & & & & \\
 \vdots &\vdots &\vdots & & & & & & \\
  & & & & -e_k & 1 & 0 & & \\
  & & & & 1 & -1 & 1 &\cdots  & \\
  & & & & 0 & 1 & -e_l' & \\
  & & & & &\vdots & & & \\
  & & & & & & & & -e_1'\\
 \end{array} \right ).
$$
If we set
$$
I=[A,\ep^2 a_1, ... , \ep^2 a_k,B,\ep^2 a_l',...,\ep^2 a_1']^t
$$
the vector which represents the integration of $\om_{\ep}$ along the divisors, we have
$$
I=Q \cdot C
$$
so 
$$
C=Q^{-1} I
$$
and this gives the parameters $(c_i)$ we were looking for, determining the K\"ahler class of $\om_{\ep}$.

The surface $\widetilde{M}$ is a ruled surface obtained by blowing-up a minimal ruled manifold. 
If the minimal model admits an extremal metric and under certain assumptions, one can construct extremal metrics 
on $\widetilde{M}$ using the gluing theory of \cite{aps}. 
However the K\"ahler classes are not the same. For example we consider the chain of curves 
that comes from resolution of a singular fiber:
$$
\xymatrix{
{}\ar@{-}[r]^{-e_1} & *+[o][F-]{}
\ar@{-}[r]^{ -e_{2}} &  *+[o][F-]{}
\ar@{--}[r] &  *+[o][F-]{}
\ar@{-}[r]^{-e_{k-1}} &  *+[o][F-]{}
\ar@{-}[r]^{-e_k} &  *+[o][F-]{}
\ar@{-}[r]^{-1} &  *+[o][F-]{}
\ar@{-}[r]^{-e'_{l}} &  *+[o][F-]{}
\ar@{-}[r]^{-e'_{l-1}} &  *+[o][F-]{}
\ar@{--}[r] &  *+[o][F-]{}
\ar@{-}[r]^{-e'_{2}} &  *+[o][F-]{}
\ar@{-}[r]^{ -e'_{1}} & 
} $$
Every curve is small except the middle one of self-intersection $-1$. On the other hand, if we used the method 
from \cite{aps}, we would have had small curves except for the 
one on the right hand side wich corresponds to the proper transform 
of the original fiber on the smooth minimal ruled surface.

\section{Applications to unstable parabolic structures}
\label{theo D}
This section is devoted to unstable parabolic structures and the proof of Theorem~\ref{general}. Let $M=\P(E)$ be a ruled surface over a Riemann surface $\Sigma$.
From \cite{acgt} if the genus of 
$\Sigma$ is greater than two, then $M$ admits a metric of constant scalar curvature in some class if and only if 
$E$ is polystable. 
On the other hand, T\o{}nnesen-Friedman has proved in \cite{tf} that $M=\P(\mathcal{O}\oplus L)$ with $\deg(L)>0$, 
if and only if $M$ admits an extremal K\"ahler metric of non-constant scalar curvature. In that case the 
bundle is decomposable and not polystable.

\begin{rmk}
Note also that Ross and Thomas have shown that for any vector bundle $E$ the K-stability of $\P(E)$ was equivalent to the polystability of $E$ \cite{rt2}. An adaptation of their argument in an equivariant context should enable to prove that $\P(E)$ admits an extremal metric of non-constant scalar curvature if and only if $E$ splits as a direct sum of stable sub-bundles, as conjectured in \cite{acgt}.
\end{rmk}

\begin{rmk}
The papers \cite{s1}, \cite{ss} and \cite{tf} confirm the Yau-Tian-Donaldson-Sz\'ekelyhidi conjecture on geometrically ruled surfaces. Together with corollary 1 from \cite{acgt}, this solves the problem of existence of extremal K\"ahler metrics on geometrically ruled surfaces in any K\"ahler class.
\end{rmk}

We now focus on parabolic ruled surfaces, performing an analogy with the previous results mentioned.
Suppose that the ruled surface is equipped with a parabolic structure $\mathcal{P}$. 
Let's recall the definition of parabolic stability from \cite{rs}:
we consider a geometrically ruled surface $\pi: M \rightarrow \Sigma$ with a parabolic structure 
given by $s$ points $A_j$ in $\Sigma$, and for each of these points a point $B_j\in\pi^{-1}(A_j)$ with 
a weight $\alpha_j\in]0,1[\cap \Q$.
\begin{dfn}
A parabolic ruled surface $M$ is parabolically stable if for every holomorphic section $S$ of $\pi$ its slope 
is strictly positive:
$$
\mu(S)=S^2+\sum_{j\notin I} \alpha_j - \sum_{j\in I} \alpha_j >0
$$
where $j\in I $ if and only if $B_j \in S$.
\end{dfn}

\begin{rmk}
If $M=\P(E)$, one can check that this definition is equivalent to the parabolic 
stability of $E$ in the sense of Mehta-Seshadri \cite{ms}.
\end{rmk}
In that case, Rollin and Singer have shown \cite{rs} that if the bundle is parabolically stable then 
there exists a scalar-flat K\"ahler metric on $Bl(M,\mathcal{P})$. 
More generally, if the surface is parabolically polystable and non-sporadic, 
then there exists a constant-scalar curvature metric on $Bl(M,\mathcal{P})$ (\cite{rs2}).
 
Given a parabolically unstable ruled surface, 
is there an extremal metric of non-constant scalar curvature on $Bl(M,\mathcal{P})$?

First of all, we show that if $Bl(M,\mathcal{P})$ admits an extremal metric of non-constant scalar curvature,
then the bundle $E$ is decomposable and one of the zero or infinity section might
destabilise $M$. Thus the situation looks like in the case studied by T\o{}nnesen-Friedman.

\begin{prop}
Let $M=\P(E)$ be a parabolic ruled surface over a Riemann surface of genus $g$
with a parabolic structure $\mathcal{P}$. Let $\frac{p_j}{q_j}$ be the weights of the parabolic structure.
Suppose that $Bl(M,\mathcal{P})$ admits an extremal metric of non-constant scalar curvature.
Then $M=\P(\mathcal{O}\oplus L)$.
Moreover, if 
$$
2-2g-\sum_j(1-\frac{1}{q_j})\leq 0
$$
the marked points of the parabolic structure all lie on the zero section or the infinity
section induced by $L$.
\end{prop}

\begin{rmk}
If 
$$
2-2g-\sum_j(1-\frac{1}{q_j}) \leq 0
$$
we can suppose that the marked point all lie on the same section. Indeed, the iterated blow-up
associated to a point on the zero section with weight $\frac{p}{q}$ is the same as the one with marked point
in the same fiber on the infinity section with weight $\frac{q-p}{q}$.
Moreover, we see that the infinity and zero sections have opposite slopes so one of them might destabilize the surface.
\end{rmk}

\begin{proof}
Let $\chi^{orb}=2-2g-\sum_j(1-\frac{1}{q_j})$.
In the case 
$\chi^{orb}>0$ the genus $g$ of $\Sigma$ is $0$ and in that case
every ruled surface is of the form $\P(\mathcal{O}\oplus L)$. We suppose that
$\chi^{orb}\leq 0$.

$Bl(M,\mathcal{P})$ admits an extremal metric of non-constant scalar curvature. Thus the extremal
vector field is not zero. It generates an action by isometries on the manifold. Using the openness theorem of
Lebrun and Simanca \cite{ls} we can suppose that the K\"ahler class of the metric is rational.
In that case, the periodicity theorem of Futaki and Mabuchi \cite{fm} implies that the action induced by the extremal metric 
is a $S^1$-action. The extremal vector field is the lift under the iterated blow-up process of a vector field 
$X$ on $M$. This vector field projects to the basis of the ruling $\Sigma$.

The projection vanishes and $X$ is vertical. Indeed, as $\chi^{orb}\leq 0$, this projection is parallel and as
$X$ lifts to the blow-ups it has to vanish somewhere, thus its projection vanishes and is zero. The restriction of $X$
on each fiber vanishes twice because it induces an $S^1$ action. The zero locus provides two sections of the ruling and $E$ splits.
Moreover, in order to preserve this vector field under the blow-up process, the marked point need to be on the zero or
infinity section.
\end{proof}

We can state a partial answer to the question of this section:
\begin{prop}
\label{unstable1}
Let $M=\P(\mathcal{O} \oplus L)$ be a ruled surface over a Riemann surface  of genus $g\geq 1$ with $L$
a holomorphic line bundle of strictly positive degree. If $\mathcal{P}$ is an unstable
parabolic structure on $M$ with every marked point on the infinity or zero section, 
then the iterated blow-up $Bl(M,\mathcal{P})$ carries an extremal K\"ahler metric
of non-constant scalar curvature.
\end{prop}

\begin{rmk}
With the work of T\o{}nnesen-Friedman in mind, and the previous result, one could expect that every unstable parabolic structure
which gives rise to extremal metric of non-constant scalar curvature on the associated iterated blow-up would lie on a surface of the
form $\P(\mathcal{O}\oplus L)$ with $L$ of degree different from zero. The theorem \ref{lastone}
proves that this is not the case.
\end{rmk}

\begin{proof}
This is an application of the main theorem in \cite{aps}.

From \cite{acgt} (see also \cite{s2}), we know that there exists extremal K\"ahler metrics of non-constant scalar curvature on $M$.
Then, the action of the extremal vector field is an $S^1$ action which rotates the fibers, fixing the zero and 
infinity sections. Indeed, the maximal compact subgroup of biholomorphisms of these surfaces is isomorphic
to $S^1$ and by Calabi's theorem the isometry group of these metrics must be isomorphic to $S^1$.
Then we can apply the result of Arezzo, Pacard and Singer to each step of the blow-up process, working modulo
this maximal torus of hamiltonian isometries.
\end{proof}

The gluing method of section \ref{secHJsing} enables to obtain more extremal metrics from unstable parabolic structures.
The end of this section consists in the proof of the following, stated as Theorem~\ref{general} in the introduction:
 
\begin{theo}
\label{lastone}
 Let $\Sigma$ be a Riemann surface of Euler characteristic $\chi$ and $L$ a line bundle of degree $d$ on $\Sigma$. 
If $\chi<0$, we suppose that $d=-\chi$ or $d\geq 1-2 \chi$. 
Then there exists an unstable parabolic structure on $\P(\mathcal{O} \oplus L)$ such that the associated iterated 
blow-up admits an extremal K\"ahler metric of non-constant scalar curvature. The K\"ahler class obtained is not 
small on every exceptional divisor.
\end{theo}

\begin{rmk}
This theorem provides extremal metrics on iterated blow-ups of parabolically unstable surfaces in different K\"ahler classes
that the one that we got in the proposition \ref{unstable1}. 
It also gives examples in the case $g=0$.
In the cases where $g=0$ or $g=1$, it also provides examples
in the case of a line bundle of degree $0$.
\end{rmk}

\begin{rmk}
The work of Sz\'ekelyhidi \cite{s1}, and \cite{ss}, show that the K\"ahler classes of the metrics constructed by T\o{}nnesen-Friedman are
exactly those which are relatively K-polystable. It might be possible to find a notion of relative parabolic stability
that corresponds to the different parabolic unstable structures considered.
\end{rmk}

We will slightly modify the construction of section \ref{surfaces2} in order to obtain more general results. 
Let $\Sigma$ be a Riemann surface. 
If $L_1$ and $L_2$ are two line bundles over $\Sigma$ of same degree, then $L_1 \otimes L_2^{-1}$ is a flat line bundle.
We will try to write every line bundle on $\Sigma$ in the following manner:
$$
\mathcal{K}^r\otimes_j[A_j]^{r-n_j} \otimes L_0
$$ 
with $L_0$ a flat line bundle.

Let $L$ be a flat line bundle on $\Sigma$. Let $\overline{\Sigma}$ be an orbifold Riemann surface topologically equivalent to $\Sigma$
with strictly negative orbifold euler characteristic .
Recall that 
$$
\pi^{orb}_1(\overline{\Sigma})\cong < (a_i,b_i)_{i=1..g}\, ,\, (l_i)_{i=1..s} \, \vert \, \Pi [a_i,b_i] \, \Pi l_i\, = l_i^{q_i}=1 >.
$$
and
$$
\pi_1(\Sigma)\cong < (a_i,b_i)_{i=1..g}\,  \vert \, \Pi [a_i,b_i] =1 >.
$$
There is a morphism:
$$
\psi: \pi^{orb}_1(\overline{\Sigma}) \rightarrow \pi_1(\Sigma)
$$
which sends $l_i$ to $1$.
As $L$ is flat, there exists a representation 
$$
\rho: \pi_1(\Sigma) \rightarrow U(1)
$$
such that 
$$
L= \widetilde{\Sigma} \times \C / \pi_1(\Sigma),
$$
where $\widetilde{\Sigma}$ is the universal cover of $\Sigma$. The action on the first factor comes from the universal covering and on the second factor from $\rho$.
Thus we have an other representation:
$$
\rho': \pi_1^{orb}(\overline{\Sigma}) \rightarrow U(1)
$$
given by 
$$
\rho'=\rho \circ \psi
$$
and a flat bundle on $\overline{\Sigma}$:
$$
L'= \D \times \C / \pi_1^{orb}(\overline{\Sigma}).
$$
We consider the orbifold K\"ahler surface 
$$
\overline {M}=\P(\mathcal{O}\oplus (\mathcal{K}_{orb}^r\otimes L')).
$$
Following an idea of T\o{}nnesen-Friedman , we see that this orbifold admits an extremal K\"ahler metric of non-constant scalar curvature.
Indeed, there is an extremal metric on $\mathcal{K}_{orb}^r$ which extends to $\P(\mathcal{O}\oplus \mathcal{K}_{orb}^r)$.
This metric and the flat metric on $L'$ provide an extremal metric on $\mathcal{K}_{orb}^r\otimes L'$ which extends similarly. 
The singularities of this orbifold are the same as the one of $\P(\mathcal{O}\oplus \mathcal{K}_{orb}^r)$.
Indeed, the choice of the representation
$$
\rho': \pi_1^{orb}(\overline{\Sigma}) \rightarrow U(1)
$$
is such that 
$$
\rho'(l_i)=1
$$
so the computations done in section \ref{sing} work in the same way. 
We can use the result of theorem \ref{generalresult} and 
we obtain a smooth ruled surface with an extremal K\"ahler metric. 
This surface is an iterated blow-up of a ruled surface which is 
$$
\P(\mathcal{O}\oplus (L_{r,(q_j)}\otimes L))
$$
where 
$$
L_{r,(q_j)}=\mathcal{K}^r\otimes_j[A_j]^{r-n_j}
$$
as in section \ref{surfaces} because the resolution and blow-down process does not affect the ``$L'$ part''.
We can state:
\begin{prop}
\label{propL}
Fix positive integers $r$ and $(q_j)_{j=1..s}$ such that for each $j$, $q_j\geq 3$ and $gcd(r,q_j)=1$. For each $j$, let 
$$p_j\, \equiv \, -r \, [q_j],\: 0<p_j<q_j\,, \: n_j=\frac{p_j+r}{q_j}.$$
Then consider a Riemann surface $\Sigma$ of genus $g$ with $s$ marked points $A_j$. The previous integers define a parabolic structure on 
$$
\check{M}=\P(\mathcal{O} \oplus (L_{r,(q_j)}\otimes L_0))
$$
with 
$$
L_{r,(q_j)}=\mathcal{K}^r\otimes_j[A_j]^{r-n_j}
$$
and $L_0$ any flat line bundle.
The parabolic structure $\mathcal{P}$ consists of the points $B_j$ in the infinity section of the ruling of $\check{M}$ over the points $A_j$ together with the weights $\dfrac{p_j}{q_j}$.
If 
$$
\chi(\Sigma)-\sum_j (1-\dfrac{1}{q_j})<0
$$
then $Bl(\check{M},\mathcal{P})$ admits an extremal K\"ahler metric of non-constant scalar curvature.
\end{prop}
We now end the proof of theorem \ref{lastone}
\begin{proof}
In order to prove theorem \ref{lastone}, it remains to show that to any Riemann surface $\Sigma$ and to each line bundle $L$ on it, we can associate an orbifold Riemann surface $\overline{\Sigma}$ defined by $\Sigma$, 
marked points $(A_j)$ and weights $q_j\geq 3$ such that 
$\chi^{orb}(\overline{\Sigma})<0$
and
$$
L=\mathcal{K}^r\otimes_j[A_j]^{r-n_j} \otimes L_0
$$
where $L_0$ is a flat line bundle.
Then the associated iterated blow-up admits an extremal metric from proposition \ref{propL}.

So let $L$ be a line bundle over $\Sigma$ and let $d$ be its degree.
We only need to show that there is a line bundle of the form
$$
L_r=\mathcal{K}^r\otimes_j[A_j]^{r-n_j}
$$
with degree $d$ on $\Sigma$, keeping in mind the euler characteristic condition.
If we manage to build such a line bundle, then $L_0=L\otimes L_r^{-1}$
is a flat line bundle and following the last proposition, we know how to obtain an iterated blow-up of
$\P(\mathcal{O}\oplus L)$
with an extremal metric.

We can suppose $d\geq 0$ because 
$$
\P(\mathcal{O}\oplus L) \cong \P(L^{-1}\oplus \mathcal{O})
$$
and 
$\deg(L)=-\deg(L^{-1})$.
We will consider three cases.

First we suppose that the genus $g$ of $\Sigma$ is $0$, that is $\chi=2$.
We consider the orbifold Riemann surface with $s\geq 4$ marked points $A_j$ with orders $q_1=q_2=..=q_{s}=3$ and we set $r=2$. 
With this choice we have 
$$
\chi^{orb}=(2-2g)- \sum (1-\dfrac{1}{q_j})<0.
$$
Then we compute the degree $d'$ of 
$$
\mathcal{K}^r\otimes [A_j]^{r-n_j}.
$$
$$
d'=r (2g-2) + \sum (r-n_j)=-2r+s(r-1)=-4+s.
$$
Then $s= 4+d$ gives the desired degree.

For $g=1$, we consider $d$ marked points of order $3$ and $r=2$.
The degree of the bundle is then equal to $d$.

It remains to study the $\chi<0$ case. 
We use the same method, $s$ marked point of order $3$.
Then $d'=-\chi r + s(r-1)$ if $r=1$ or $r=2$.
$r=1$ gives $d'=-\chi$ and $r=2$ gives $d'=s-2\chi$, which give the restriction stated in \ref{lastone}.

It is not difficult to see that the surfaces considered in the theorem \ref{theoB} with the parabolic structure of 
section \ref{surfaces} are not parabolically stable. 
Indeed, if we consider $\P(\mathcal{O}\oplus L)$, and if we denote $S_0$ and $S_{\infty}$ the zero and infinity sections, 
following \cite{rs} we have 
$$
\mu(S_{\infty})=S_{\infty}^2-\sum_j \alpha_j
$$
and
$$
S_{\infty}^2=deg(\mathcal{O}\oplus L)-2deg(L)=-deg(L).
$$
So 
$$
\mu(S_{\infty})=-(-r\chi + \sum_j (r-n_j))-\sum \dfrac{p_j}{q_j}
$$
$$
\mu(S_{\infty})=r(\chi -\sum_j (1-\dfrac{1}{q_j}))
$$
Thus
$$
\mu(S_{\infty})=r \chi^{orb}<0
$$
And $S_{\infty}$ destabilises $\check{M}$.
\end{proof}

\begin{rmk}
If we consider more general constructions, we have
$$
d'=(-\chi  + s )r - \sum_j n_j
$$
with $s$ marked points. However the left part grows linearly in $r$ while if 
we write $r=q_j r_j-p_j$ the right part decreeses as $-r_j$ so we do not expect to obtain smaller degrees with this method.
\end{rmk}

\section{examples}
We will give two examples which lead to the results Theorem~\ref{cp2} and Theorem~\ref{cp1 t2}.

\subsection{First example}
\label{toriccp2}
Here we will use the extremal metrics on weighted projective spaces constructed by Bryant \cite{b}.
We also refer to the work of Abreu \cite{a}. In his paper, he constructed extremal 
K\"ahler metrics with non-constant scalar curvature on weighted projective spaces
$$
\C\P^2_{a,b,c} := \C^3 / \C^*
$$
where the action is
$$
\forall t \in \C^* , t.(x,y,z):=(t^a x, t^b y, t^c z).
$$
In particular, we can endow $\C\P^2_{1,2,3}$ with an extremal metric. This orbifold has two isolated singularities: $A_{1,2}$ at $[0,1,0]$ and $A_{2,3}$ at $[0,0,1]$. Thus we can endow a minimal resolution $X$ of $\C\P^2_{1,2,3}$ with an extremal K\"ahler metric. Following Fulton (\cite{f}), we use the fan description of these toric manifolds. The fan associated to $\C\P^2_{1,2,3}$ is represented on Figure 1. The minimal resolution is represented Figure 2. The fan of figure 2 is also associated to a three times iterated blow-up of $\C\P^2$.
\begin{figure}[htbp]
\psset{unit=0.85cm}
\begin{pspicture}(-4.5,-7.5)(4.5,0)
$$
\xymatrix @M=0mm{
\bullet & \bullet & \bullet & \bullet & \bullet \\
\bullet & \bullet & \bullet & \bullet & \bullet \\
\bullet & \bullet & \bullet & \bullet & \bullet \\
 \bullet & \bullet & \bullet\ar[u]\ar[r]\ar[llddd] & \bullet & \bullet \\
\bullet & \bullet & \bullet & \bullet & \bullet \\
\bullet & \bullet & \bullet & \bullet & \bullet \\
\bullet & \bullet & \bullet & \bullet & \bullet 
}
$$
\end{pspicture}
\caption{$\C\P^2_{1,2,3}$}
\end{figure}
\begin{figure}[htbp]
\psset{unit=0.85cm}
\begin{pspicture}(-4.5,-7.5)(4.5,0)
$$
\xymatrix @M=0mm{
\bullet & \bullet & \bullet & \bullet & \bullet \\
\bullet & \bullet & \bullet & \bullet & \bullet \\
\bullet & \bullet & \bullet & \bullet & \bullet \\
 \bullet & \bullet & \bullet\ar[u]\ar[r]\ar[llddd]\ar[ld]\ar[ldd]\ar[d] & \bullet & \bullet \\
\bullet & \bullet & \bullet & \bullet & \bullet \\
\bullet & \bullet & \bullet & \bullet & \bullet \\
\bullet & \bullet & \bullet & \bullet & \bullet 
}
$$
\end{pspicture}
\caption{Minimal resolution}
\end{figure}
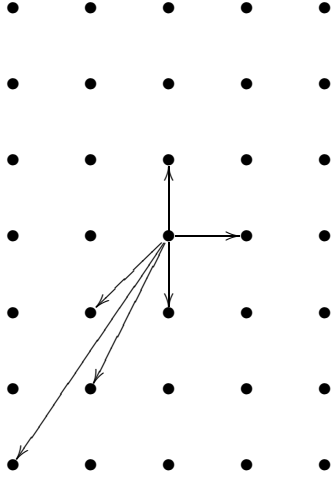

Now we describe the K\"ahler classes which arise this way. The singular homology group $H_2(X,\Z)$ 
is generated by $H,E_1,E_2$ and $E_3$ where $H$ is the proper transform of a hyperplane in $\C\P^2$ 
and the $E_i'$s are the successive exceptional divisors. If the first blow-up is done on a point of $H$, 
we get the following chain of curves:
$$
\xymatrix{
{}\ar@{-}[r]^{H}_{-2} & *+[o][F-]{}
\ar@{-}[r]^{E_3}_{-1} &  *+[o][F-]{}
\ar@{-}[r]^{E_2}_{-2} &  *+[o][F-]{}
\ar@{-}[r]^{ E_1}_{-2} &   .
}
$$
$E_1$ and $E_2$ come from the resolution of the $A_{2,3}$ singularity and $H$ from the resolution of the $A_{1,2}$ singularity of $\C\P^2_{1,2,3}$. So these divisors are small. $E_3$ comes from $\C\P^2_{1,2,3}$ as the pull back of the line $\overline{H}$ joining the two singularities, and integrating the metric on it will be related to the volume of $\C\P^2_{1,2,3}$, which can be chosen arbitrarily. Indeed, the construction of the metric by Abreu is done on a space $\C\P^2_{[1,2,3]}$ and then pulled back to $\C\P^2_{1,2,3}$ by a map 
$$
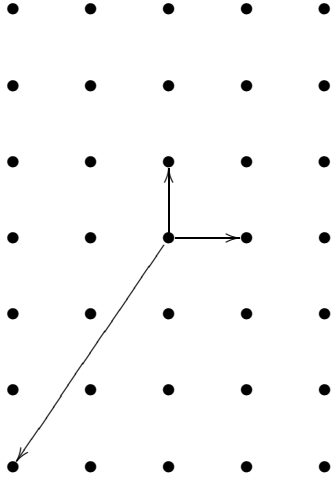
p:\C\P^2_{1,2,3} \rightarrow \C\P^2_{[1,2,3]}.
$$
As $\C\P^2_{[1,2,3]}$ is diffeomorphic to $\C\P^2$ (it is even biholomorphic but not as an orbifold), its homology group $H_2(\C\P^2_{[1,2,3]},\Z)$ is one dimensional and evaluating the metric on $p(\overline{H})$ will give a constant proportional to the volume. So it is for $\overline{H}$.
Next, following the method of Section~\ref{classes}, we compute the intersection form $Q$ 
in the basis $H$, $E_3$, $E_2$ and $E_1$
$$
Q=\left ( \begin {array} {cccc} 
 -2 & 1 & 0 & 0 \\
 1 & -1 & 1 & 0 \\
 0 & 1 & -2 & 1 \\
 0 & 0 & 1 & -2 \\
 \end{array} \right ).
$$
Then
$$
I=[\ep^2 a_3,a,\ep^2 a_2,\ep^2 a_1]^t
$$
with $a$ and the $a_i$ arbitrary positive numbers and $\ep$ small enough. The computation of $Q^{-1}\cdot I$ gives 
the K\"ahler class
$$
 (3a+\ep^2(a_3+2a_2+a_1))PD(H) + (2a+\ep^2(a_3+a_2))PD(E_1)
$$
$$
+(4a+\ep^2(2a_3+2a_2+a_1))PD(E_2)+(6a+\ep^2(3a_3+4a_2+2a_1)) PD(E_3).
$$
It proves Theorem~\ref{cp2}.

\subsection{Second example}
We now consider the orbifold Riemann surface of genus $1$ with a singularity of order $3$. 
In this case $\chi^{orb}<0$ and we can use the results of Corollary~\ref{orbiextrem} with $r=1$. 
The associated orbifold ruled surface has two singular points of order $3$ and from Proposition~\ref{surface2} 
we know that a minimal resolution is a three times iterated blow-up of the surface $\P(\mathcal{O} \oplus L)$ 
over $\Sigma_1\simeq\T^2$. Here $L=\mathcal{O}$.
Thus we get an extremal K\"ahler metric on a three times blow-up of $\C\P^1\times \T^2$. 
The iterated blow-up can be made more precise. The first point to be blown-up is the point on the zero section 
above the marked point of $\overline{\Sigma}$. The iterated blow-up replace the fiber $F$ by the chain of curves:
$$
\xymatrix{
{}\ar@{-}[r]^{-2} & *+[o][F-]{}
\ar@{-}[r]^{-2} &  *+[o][F-]{}
\ar@{-}[r]^{-1} &  *+[o][F-]{}
\ar@{-}[r]^{ -3} & 
}
$$
with the $-3$ self-intersection curve corresponding to the proper transform of the fiber $F$ above the first blown-up point.
This ends the proof of the Theorem~\ref{cp1 t2} stated in the introduction.


\begin{thebibliography}{BPV}

\bibitem{a} \textit{M.Abreu},
K\"ahler metrics on toric orbifolds, J. Differential Geom. \textbf{58} (2001), no. 1, 151-187.

%\bibitem{at} \textit{V.Apostolov and C.W. T\o{}nnesen-Friedman},
%A remark on K\"ahler metrics of constant scalar curvature on ruled complex surfaces,
%Bull. London Math. Soc. \textbf{38} (2006), no. 3, 494–500. 

\bibitem{ap1} \textit{C.Arezzo and F.Pacard},
Blowing up and desingularising constant scalar curvature K\"ahler manifolds,
Acta Math. \textbf{196} (2006), no. 2, 179-228. 

\bibitem{ap2} \textit{C.Arezzo and F.Pacard}, 
Blowing up constant scalar curvature K\"ahler manifolds II, Ann. of Math. (2) 
\textbf{170} (2009), no. 2, 685-738. 

\bibitem{aps} \textit{C.Arezzo, F.Pacard and M.Singer},
Extremal Metrics on blow ups, Duke Math. J. Volume \textbf{157}, Number 1 (2011), 1-51. 

\bibitem{acgt} \textit{V.Apostolov, D.M.J. Calderbank, P.Gauduchon, C.W. T\o{}nnesen-Friedman},
Extremal K\"ahler metrics on projective bundles over a curve,  Adv. Math. \textbf{227} (2011), 2385-2424.

\bibitem{bpv} \textit{W.Barth, C.Peters and A.Van de Ven},
Compact complexe surface, Springer 1984.

\bibitem{b} \textit{R.Bryant},
Bochner-K\"ahler metrics,
J.Amer.Math.Soc. \textbf{14} (2001), 623-715.

%\bibitem{bb} \textit{D.Burns and P. de Bartolomeis}, 
%Stability of vector bundles and extremal metrics,
%Invent. Math. \textbf{92} (1988), 403-407.

\bibitem{c1} \textit{E.Calabi}, 
Extremal K\"ahler metrics,
"Seminar on Differential Geometry" (ed.S.-T. Yau), Princeton, 1982.

\bibitem{c2} \textit{E.Calabi},
Extremal K\"ahler metrics II, "Differential Geometry and Complex Analysis" (ed. I. Chavel et H.M. Farkas), 
Springer-Verlag, 1985.

\bibitem{cs} \textit{D.Calderbank and M.Singer},
Einstein metrics and complex singularities,
Invent.Math.\textbf{156} (2004),no2, 405-443.

\bibitem{d} \textit{S.K. Donaldson},
Scalar curvature and stability of toric varieties, J. Differential Geom. 
\textbf{62} (2002), no. 2, 289-349. 
 
\bibitem{f} \textit{W.Fulton},
Toric Varieties, Princeton University Press 1993.

\bibitem{fm} \textit{A.Futaki and T.Mabuchi},
Bilinear forms and extremal K\"ahler vector fields associated with K\"ahler classes, 
Math. Ann. \textbf{301} (1995), no. 2, 199-210.  .

\bibitem{gbook} \textit{P. Gauduchon}, 
Calabi's extremal metrics: An elementary introduction, 
book in preparation (2011).

\bibitem{g} \textit{R.C.Gunning},
Lectures on Riemann Surfaces,
Princeton mathematical notes. Princeton University Press 1966.

\bibitem{j} \textit{D.D.Joyce},
Compact manifolds with special holonomy,
Oxford University press 2000.

\bibitem{ls} \textit{C.Lebrun and S.Simanca},
Extremal K\"ahler metrics and complexe deformation theory,
Geom. Func. Anal.\textbf{4} (1994), 298-336.

%\bibitem{l} \textit{M.levine}
%A remark on extremal K\"ahler metrics,
%J. Differential Geom. \textbf{21} (1985), 73-77.

\bibitem{le} \textit{E.Legendre},
Toric geometry of convex quadrilaterals,
Journal of Symplectic Geometry 9 (2011), pp.343-385.

\bibitem{li} \textit{A.Lichnerowicz},
G\'eom\'etrie des groupes de transformation,
Travaux et recherches math\'ematiques \textbf{3}, Dunod (1958).

%\bibitem{ma} \textit{Y.Matsushima},
%Sur la structure du groupes d'hom\'eomorphismes analytiques d'une certaine vari\'et\'e K\"ahl\'erienne, 
%Nagoya Maths. J.\textbf{11} (1957), 145-150.

\bibitem{mab} \textit{T.Mabuchi}
K-stability of constant scalar curvature polarization, arXiv:0812.4093v2.

\bibitem{ms} \textit{V.B.Mehta and C.S.Seshadri}
Moduli of vector bundles over curves with parabolic structures,
Math.Ann. \textbf{248} (1980), 205-239.

\bibitem{rs} \textit{Y.Rollin and M.Singer},
Non-minimal scalar-flat K\"ahler surfaces and parabolic stability,
Invent.Math. \textbf{162} (2005), 235-270.

\bibitem{rs2} \textit{Y.Rollin and M.Singer}, 
Constant scalar curvature K\"ahler surfaces and parabolic polystability, 
J. Geom. Anal. \textbf{19} (2009), no. 1, 107-136. 

\bibitem{rt2} \textit{J.Ross and R.Thomas}
 An obstruction to the existence of constant scalar curvature Kahler metrics, Jour. Diff. Geom. 72, 2006, 429-466.

\bibitem{rt} \textit{J. Ross and R.Thomas}, 
Weighted projective embeddings, stability of orbifolds and constant scalar curvature K\"ahler metrics, 
Jour. Diff. Geom. 88, 2011, 109-160.

\bibitem{s} \textit{G. Sz\'ekelyhidi} 
On blowing up extremal K\"ahler manifolds,
Duke Math. J., 161 (2012) n. 8, 1411-1453.

\bibitem{s2} \textit{G. Sz\'ekelyhidi},
The Calabi functional on a ruled surface, Ann. Sci. \'Ec. Norm. Sup\'er. 42 (2009), 837-856.

\bibitem{s1} \textit{G. Sz\'ekelyhidi},
Extremal metrics and K-stability, 
Bull. Lond. Math. Soc. \textbf{39} (2007), no. 1, 76-84. 

\bibitem{ss} \textit{G. Sz\'ekelyhidi and J.Stoppa}, 
Relative K-stability of extremal metrics, 
J. Eur. Math. Soc. 13 (2011) n. 4, 899-909.

\bibitem{t} \textit{G.Tian},  
K\"ahler-Einstein metrics with positive scalar curvature, 
Invent. Math. \textbf{130} (1997), 1-37.

\bibitem{tf} \textit{C.W. T\o{}nnesen-Friedman}, 
Extremal K\"ahler metrics on Ruled Surfaces,  
J. reine angew. Math. 502 (1998) 175-197.

\bibitem{y} \textit{S.Y. Yau}, 
Open problems in Geometry,
 in 'Differential geometry: partial differential equations on manifolds'(Los Angeles, CA, 1990) ,
Proc.Sympos.Pure Math., AMS \textbf{54}, 1-28.

\end{thebibliography}
\end{document}